\documentclass{amsart}

\usepackage{etex}
\usepackage[table,svgnames]{xcolor}
\usepackage{tikz}
\usetikzlibrary{arrows,chains,matrix,positioning,scopes,calc,decorations.pathmorphing,intersections}
\usepackage{frcursive}
\usepackage{amssymb}
\usepackage{hyperref}
\usepackage{enumerate}
\usepackage{hhline}
\usepackage{graphicx}
\usepackage{multirow}
\usepackage{float}
\usepackage{stmaryrd}
\usepackage{tikz-cd}
\usepackage{afterpage}
\usepackage{lscape}

\usepackage{todonotes}

\newtheorem{theorem}{Theorem}[section]

\newtheorem{prop}[theorem]{Proposition}
\newtheorem{conj}[theorem]{Conjecture}

\newtheorem*{conjintro2}{Conjecture}

\theoremstyle{definition}
\newtheorem{definition}[theorem]{Definition}
\newtheorem{example}[theorem]{Example}

\theoremstyle{remark}

\numberwithin{equation}{section}

\newcommand{\N}{\mathbb N}
\newcommand{\Z}{\mathbb Z}
\newcommand{\Q}{\mathbb Q}

\newcommand{\C}{\mathbb C}

\def\sgn{{\rm sgn}}

\def\Fr{{\rm Fr}}
\def\rec{{\rm rec}}
\def\cusp{{\rm cusp}}
\def\Hom{{\rm Hom}}
\def\proj{{\bf p}}
\def\Sc{{\bf Sc}}

\def\temp{{\rm temp}}
\def\Ad{{\rm Ad}}
\def\diag{{\rm diag}}
\def\St{{\rm St}}
\def\st{{\rm st}}
\def\SO{{\rm SO}}
\def\O{{\rm O}}
\def\Sp{{\rm Sp}}
\def\GL{{\rm GL}}
\def\SL{{\rm SL}}
\def\Irr{{\rm Irr}}
\def\Sil{\textrm{\footnotesize\cursive\slshape Si}}
\def\Scl{\textrm{\footnotesize\cursive\slshape Sc}}
\def\wj{\textrm{\footnotesize\cursive\slshape j}}
\def\wi{\textrm{\footnotesize\cursive\slshape i}}
\def\restriction#1#2{\mathchoice
              {\setbox1\hbox{${\displaystyle #1}_{\scriptstyle #2}$}
              \restrictionaux{#1}{#2}}
              {\setbox1\hbox{${\textstyle #1}_{\scriptstyle #2}$}
              \restrictionaux{#1}{#2}}
              {\setbox1\hbox{${\scriptstyle #1}_{\scriptscriptstyle #2}$}
              \restrictionaux{#1}{#2}}
              {\setbox1\hbox{${\scriptscriptstyle #1}_{\scriptscriptstyle #2}$}
              \restrictionaux{#1}{#2}}}
\def\restrictionaux#1#2{{#1\,\smash{\vrule height .8\ht1 depth .85\dp1}}_{\,#2}} 

\usepackage[citestyle=alphabetic,bibstyle=alphabetic,maxnames=4,backend=bibtex]{biblatex}
\DeclareFieldFormat{postnote}{#1}
\DeclareFieldFormat{multipostnote}{#1}

\newcommand{\yslant}{0.5}
\newcommand{\xslant}{-0.6}
\tikzset{snake arrow/.style=
{->,
decorate,
decoration={snake,amplitude=7pt,segment length=20pt,post length=5pt}},
}

\definecolor{gris1}{gray}{0.9}
\definecolor{gris2}{gray}{0.75}
\definecolor{gris3}{gray}{0.6}
\begin{document}

\title[Proof of the ABPS conjecture for split classical groups]{Proof of the Aubert-Baum-Plymen-Solleveld conjecture for split classical groups}

\author[A. Moussaoui]{Ahmed Moussaoui}

\address{Pacific Institute for the Mathematical Sciences and Department of Mathematics and Statistics, University of Calgary, 
2500 University Drive NW, Calgary, Alberta, Canada, T2N 1N4}

\email{ahmed.moussaoui@ucalgary.ca, ahmedmoussaouimath@gmail.com}
\thanks{The author would like to thank the anonymous referee, Anne-Marie Aubert, Mounir Hajli, Bin Xu and especially Clifton Cunningham for several comments and corrections on a preliminary version.\\ This article was written as a PIMS Post Doctoral Fellow.}


\date{}

\begin{abstract}
In this paper we prove the Aubert-Baum-Plymen-Solleveld conjecture for the split classical groups and establish the connection with the Langlands correspondence. To do this, we review the notion of cuspidality for enhanced Langlands parameters and also review the notion of cuspidal support for enhanced Langlands parameters for split classical groups, both introduced by the author in earlier work.
\end{abstract}

\maketitle
\setcounter{tocdepth}{2}     
\setcounter{secnumdepth}{2}  
\tableofcontents

\section*{Introduction}
Let $G$ be a connected reductive $p$-adic group and $\Irr(G)$ be the set of (classes of) smooth irreducible complex representations of $G$. On the one hand, the Bernstein decomposition gives a way to study $\Irr(G)$ in terms of parabolic induction. On the other hand, the local Langlands correspondence predicts a decomposition of $\Irr(G)$ into finite subsets. It is natural to ask what is the relation between these two decompositions? We are particularly interested in the question of what are the Langlands parameters for supercuspidal representations (see Definition~\ref{defcuspS}) and how to define cuspidal support for (enhanced) Langlands parameters (see Theorem~\ref{theoremesupportcuspidal}).  

In this paper we prove the Aubert-Baum-Plymen-Solleveld conjecture for the split classical groups. This conjecture was also proved by Solleveld in \cite{Solleveld:2012aa} using different arguments, which are discussed in Section~\ref{sec:ABPSconj}. However, in that proof there is no link with the Langlands correspondence. The proof presented here makes clear the relation between the ABPS conjecture and the Langlands correspondence. Roughly speaking, this is done by studying the link between the Langlands correspondence and parabolic induction. 
This requires a quick overview of \cite{Moussaoui:2015}. In fact, the main motivation for \cite{Moussaoui:2015} was the study of the Aubert-Baum-Plymen-Solleveld conjecture. In particular, we note that our constructions fit naturally with the work of Haines \cite{Haines:2013aa} on the stable Bernstein centre, especially Conjecture~\ref{conjindLLC}, regarding the compatibility of parabolic induction and the local Langlands correspondence.

In order to state the main result of this paper, we briefly review the Aubert-Baum-Plymen-Solleveld conjecture, beginning with what is commonly referred to as an extended quotient.

Let $T$ be a complex affine variety and $\Gamma$ be a finite group acting on $T$. For all $t \in T$, let $\Gamma_t =\{ \gamma \in \Gamma \mid \gamma \cdot t = t \}$ be the stabilizer of $t$ in $\Gamma$. 
The group $\Gamma$ acts on $$Y=\{(t,\rho) \mid t \in T, \rho \in \Irr(\Gamma_t)\}$$ by $$\alpha \cdot (t,\rho) =(\alpha \cdot t, \alpha^* \rho), \,\, \alpha \in \Gamma, \, (t,\rho) \in Y,$$ where $\alpha^* \rho \in \Irr(\Gamma_{\alpha \cdot t})$ is defined by, $(\alpha^* \rho)(\gamma)=\rho(\alpha \gamma \alpha^{-1})$, for all $\gamma \in \Gamma_{\alpha \cdot t}$.

The \emph{spectral extended quotient of $T$ by $\Gamma$} is the quotient $Y/\Gamma$ and it is denoted by $T \sslash \widehat{\Gamma}$.

Note that the projection map on the first coordinate $Y \longrightarrow T$ is $\Gamma$-equivariant and this defines a projection map $T \sslash \widehat{\Gamma} \longrightarrow T /\Gamma$.

We now recall the Bernstein decomposition; see \cite[2.10,2.13]{Bernstein:1984aa} and \cite[VI.7.1,VI.7.2,VI.10.3]{Renard:2010aa} for more detail.
Let $G$ be a connected reductive group defined and split over a $p$-adic field. 
We denote by $i_P^G$ and $r_P^G$ the parabolic induction and Jacquet functors, respectively. 
Let $\pi$ be an irreducible smooth representation of $G$. 
Let $P$ be a parabolic subgroup of $G$ with Levi factor $M$ such that $r_P^G(\pi) \neq 0$ and minimal for this property. 
Let $\sigma$ be an irreducible subquotient of $r_P^G(\pi)$. 
Then $\sigma$ is an irreducible supercuspidal representation of $M$. 
Moreover if $(M',\sigma')$ is a pair which arises in the same way for another parabolic subgroup $P'$, then there exists $g \in G$ such that $M'={}^g M$ and $\sigma \simeq {}^g \sigma'$. 
The $G$-conjugacy class of the pair $(M,\sigma)$ is called the \emph{cuspidal support} of $\pi$. 
There are two equivalence relations on the set of pairs $(M,\sigma)$ where $M$ is a Levi subgroup of $G$ and $\sigma$ is an irreducible supercuspidal representation of $M$:
conjugation by $G$ on these pairs and conjugation by $G$ up to an unramified character. More precisely, if $(M_1,\sigma_1)$ and $(M_2,\sigma_2)$ are two such pairs, then $(M_1,\sigma_1)$ is said to be associated (resp. inertially equivalent) to $(M_2,\sigma_2)$ if there exists $g \in G$ such that ${}^gM_1=M_2$ and $\sigma_1^g \simeq \sigma_2$ (resp. there exist $g \in G$ and an unramified character $\chi_2 \in \mathcal{X}(M_2)$ such that ${}^gM_1=M_2$ and $\sigma_1^g \simeq \sigma_2 \chi_2$). 
Let $\Omega(G)$ and $\mathcal{B}(G)$ be the set of associated equivalence classes (resp. inertial equivalence classes) of pairs $(M,\sigma)$ where $M$ is a Levi subgroup of $G$ and $\sigma$ is an irreducible supercuspidal representation of $M$. 
Because the group $\mathcal{X}(M)$ of unramified characters of $M$ has a torus structure, we can associate the following to each $\mathfrak{s}=[M,\sigma] \in \mathcal{B}(G)$:
\begin{itemize}
\item a torus $T_{\mathfrak{s}}=\{\sigma \otimes \chi, \chi \in \mathcal{X}(M)\}$;
\item a finite group $W_{\mathfrak{s}}=\left\{w \in N_G(M)/M \mid \exists \chi \in \mathcal{X}(M), \sigma^{w} \simeq \sigma \otimes \chi \right\}$;
\item an action of $W_{\mathfrak{s}}$ on $T_{\mathfrak{s}}$.
\end{itemize}
Notice that $T_{\mathfrak{s}}$ becomes a torus after the choice of a base point $(M,\sigma)$. The fiber over $\mathfrak{s} \in \mathcal{B}(G)$ under the projection map $\Omega(G) \twoheadrightarrow \mathcal{B}(G)$ is identified with the quotient $T_{\mathfrak{s}}/W_{\mathfrak{s}}$. The Bernstein decomposition of the set of irreducible representations of $G$ is a partition of $\Irr(G)$ indexed by $\mathcal{B}(G)$: $$\Irr(G)=\bigsqcup_{\mathfrak{s} \in \mathcal{B}(G)} \Irr(G)_{\mathfrak{s}}.$$
Moreover, the cuspidal support map restricts on each piece to a map $\Sc: \Irr(G)_{\mathfrak{s}} \longrightarrow T_{\mathfrak{s}} / W_{\mathfrak{s}}$.
The benefit of this extended quotient is the following conjecture, which predicts that we can recover $\Irr(G)_{\mathfrak{s}}$ from the data associated to $\mathfrak{s}$ described above.
\begin{conjintro2}[Aubert-Baum-Plymen-Solleveld]
For each $\mathfrak{s} \in \mathcal{B}(G)$, there exists a bijection 
\[
\mu_{\mathfrak{s}} :  \Irr(G)_{\mathfrak{s}} \longrightarrow T_{\mathfrak{s}} \sslash \widehat{W_{\mathfrak{s}}}.
\]
  In general the following diagram is not commutative \\
\begin{center}
\begin{tikzcd}
  \Irr(G)_{\mathfrak{s}} \arrow[swap]{rd}{\Sc} \arrow[leftrightarrow]{rr}{\mu_{\mathfrak{s}}} & & T_{\mathfrak{s}} \sslash \widehat{W_{\mathfrak{s}}} \arrow{ld}{\proj_{\mathfrak{s}}} \\
   & T_{\mathfrak{s}} / W_{\mathfrak{s}} &
  \end{tikzcd}
\end{center}
but by precomposing the projection on the right with certain cocharacters of $T_{\mathfrak{s}}$, called \emph{correcting cocharacters}, then this diagram is commutative.
\end{conjintro2}

In \cite{Aubert:2014ab}, in the case where the Levi subgroup defining the inertial pair is a maximal torus of a split group, the authors show that the correcting cocharacter associated to $[t,\rho] \in T_{\mathfrak{s}} \sslash \widehat{W_{\mathfrak{s}}}$ is $\phi_{\pi}(1,\diag(t,t^{-1}))$ where $\pi=\mu_{\mathfrak{s}}^{-1}[t,\rho]$ and $\phi_\pi$ is the Langlands parameter of the representation $\pi$. In this paper, if $\mathfrak{s}=[M,\sigma]$ we show a more general formula for the correcting cocharacter of $[t,\rho] \in T_{\mathfrak{s}} \sslash \widehat{W_{\mathfrak{s}}}$, namely: $\phi_{\pi}(1,\diag(t,t^{-1}))/\phi_{\sigma}(1,\diag(t,t^{-1}))$ where $\pi=\mu_{\mathfrak{s}}^{-1}[t,\rho]$, $\phi_{\pi}$ and $\phi_{\sigma}$ are the Langlands parameters of the representations $\pi$ and $\sigma$ respectively.

Here we prove the ABPS conjecture for split classical groups by establishing a \emph{Galois version} of the ABPS conjecture, obtained by replacing the representations with their (enhanced) Langlands parameters.  To do this we use \cite{Moussaoui:2015} which shows how to convert the supercuspidality of the representation into a condition on the corresponding (enhanced) Langlands parameter. 

The article is organized as follows. 
In Section~\ref{sec:Springer}, we review the generalized Springer correspondence which will be a tool for the next steps. Here we give the examples of $\GL_n$, $\Sp_6$ and $\SO_4$.
In Section~\ref{sec:Langlands} we briefly recall the local Langlands correspondence for split groups, paying special attention to the case of split classical groups. Then we recall the notion of cuspidal enhanced Langlands parameters from \cite{Moussaoui:2015} and we explain how to construct the cuspidal support of an enhanced Langlands parameter in the case of split classical groups.
Finally, in Section~\ref{sec:proof}, after finding the predicted correcting cocharacters, we prove the ABPS conjecture for split classical groups.  We give a concrete example to illustrate it in the case of $\Sp_4(F)$.\\

\section{Springer correspondence}\label{sec:Springer}

Let $H$ be a complex reductive algebraic group and consider the set $$\mathcal{U}_H^e=\{ (\mathcal{C}_u^{H},\eta) \mid u \in H \,\,\text{unipotent}, \,\, \eta \in \Irr(A_H(u))\},$$ where $\mathcal{C}_u^H$ denotes the $H$-conjugacy class of $u$ and $A_H(u)=Z_H(u)/Z_H(u)^{\circ}$ with $Z_H(u)$ the centralizer of $u$ in $H$. We denote the Weyl group of $H$ by $W_H=N_H(T)/T$ with $T$ a maximal torus of $H$. Suppose from here that $H$ is connected.

\begin{example}
Let $n \geqslant 1$ be an integer and consider the group $H=\GL_n(\C)$. For any element $u \in H$, the group $A_H(u)$ is trivial. A maximal torus $T$ of $H$ is the group of diagonal matrices and the Weyl group of $H$ is $W_H \simeq \mathfrak{S}_n$, the symmetric group over $n$ letters. Moreover, by the Jordan classification, the set $\mathcal{U}_H^e$ is parametrized by $\mathcal{P}(n)$, the set of partitions of $n$, as follows : $$ \begin{array}{ccc}
\mathcal{P}(n) & \longrightarrow & \mathcal{U}_H^e \\
(p_1,\ldots,p_r) & \longmapsto & \left(\begin{pmatrix}
J_{p_1} &  &  \\ 
 & \ddots  &  \\ 
 &  & J_{p_r}
\end{pmatrix}, \mathrm{triv} \right)
\end{array}  $$ with $$J_d=\begin{pmatrix}
1 & 1 &  & \\ 
 & 1  & 1 &  \\ 
 & & \ddots & \ddots \\
 &  && 1 & 1\\
 & && & 1
\end{pmatrix} \in \GL_d(\C)$$
By the theory of Young diagrams, irreducible representations of $W_H \simeq \mathfrak{S}_n$ are parametrized by $\mathcal{P}(n)$. This gives a bijection between $\Irr(W_H)$ and $\mathcal{U}_H^e$.
\end{example}

\subsection{Ordinary Springer correspondence}

In general, when $H$ is different from $\GL_n$, we do not have a bijection between $\Irr(W_H)$ and $\mathcal{U}_H^e$ but there is an embedding $\Irr(W_H) \hookrightarrow \mathcal{U}_H^e$; this embedding is called the ordinary Springer correspondence. It was studied by Springer during the 1970s in \cite{Springer:1978aa}. The ordinary Springer correspondence for $H$ relates two different objects in nature: irreducible representations of the Weyl group of $H$ and pairs ($\mathcal{C}^H_u, \eta)$, where $\mathcal{C}^H_u$ is a unipotent orbit in $H$ and $\eta$ is an irreducible representation of $A_H(u)$.

The Springer correspondence can be described combinatorially.

\begin{example}\label{ordspringer}
Recall that the unipotent classes of $H=\Sp_{2n}(\C)$ are in bijection with partitions of $2n$ for which the odd parts have even multiplicity. 
The Weyl group $W_H$ of $H$ is isomorphic to $\mathfrak{S}_n \rtimes (\Z/2\Z)^n$ and its irreducible representations are in bijection with the set of bipartitions of $n$, {\it i.e.}, the pairs $(\alpha,\beta)$ where $\alpha, \beta$ are partitions (perhaps trivial) such that $|\alpha|+|\beta|=n$. For instance, the trivial representation corresponds to the partition $(n,0)$ while the sign representation corresponds to the partition $(0,1^n)$. \\
See Table~\ref{table:Sp6-1} for the case $H=\Sp_6(\C)$. 
\end{example}

\begin{table}[H]
\begin{center}
\renewcommand\arraystretch{1.2}
$$\begin{array}{|c|c|c|c|}
\hhline{----}
u & A_H(u) & \Irr(A_H(u)) & \Irr(W_H^L) \\ \hhline{====}
 \multirow{2}*{$(6)$} & \multirow{2}*{$\Z/2\Z$} & \cellcolor{gris1} 1 & \cellcolor{gris1} \rho_{(3,\varnothing)} \\ \hhline{~~--}
 &  & \zeta & \\ \hhline{----}
\multirow{4}*{$ (4,2)$} & \multirow{4}*{$(\Z/2\Z)^2$} &\cellcolor{gris1}  1 \boxtimes 1 &\cellcolor{gris1}  \rho_{(2,1)} \\ \hhline{~~--}
 & & \cellcolor{gris1}  \zeta \boxtimes \zeta  & \cellcolor{gris1} \rho_{(\varnothing,3)}  \\ \hhline{~~--}
 &  & 1 \boxtimes \zeta &  \\ \hhline{~~--}
&  & \zeta \boxtimes 1 &  \\ \hhline{----}
\multirow{2}*{$ (4,1^2)$} & \multirow{2}*{$\Z/2\Z$} & \cellcolor{gris1}  1 & \cellcolor{gris1}  \rho_{((2,1),\varnothing)} \\ \hhline{~~--}
 &  & \zeta & \\ \hhline{----}
 (3^2) & \{1\} & \cellcolor{gris1}  1 & \cellcolor{gris1} \rho_{(1,2)} \\ \hhline{----}
\multirow{2}*{$(2^3)$} & \multirow{2}*{$\Z/2\Z$} & \cellcolor{gris1} 1& \cellcolor{gris1}  \rho_{(1^2,1)}  \\ \hhline{~~--}
& & \zeta  & \\ \hhline{----}
\multirow{2}*{$ (2^2,1^2)$} & \multirow{2}*{$\Z/2\Z$} &\cellcolor{gris1}  1  & \cellcolor{gris1} \rho_{(1,1^2)} \\ \hhline{~~--}
 & &\cellcolor{gris1}  \zeta &\cellcolor{gris1}  \rho_{(\varnothing,(2,1))} \\ \hhline{----}
\multirow{2}*{$(2,1^4)$} & \multirow{2}*{$\Z/2\Z$} &\cellcolor{gris1} 1  &\cellcolor{gris1}  \rho_{(1^3,\varnothing)}  \\ \hhline{~~--}
&  & \zeta  &\\ \hhline{----}
(1^6) & \{1\} &\cellcolor{gris1}  1  & \cellcolor{gris1} \rho_{(\varnothing,1^3)} \\ \hhline{----}
\end{array} $$
\end{center}
\caption{Springer correspondence for $\Sp_6(\C)$}
\label{table:Sp6-1}
\end{table}%

As we can see, some pairs $(\mathcal{C}_{u}^{H},\eta)$ do not appear in the ordinary Springer correspondence. The parametrization of such pairs is the purpose of the next section.

\subsection{Generalized Springer correspondence}\label{sec:GSC}

One can ask how to describe elements in $\mathcal{U}_H^e$ which are not in the image of the ordinary Springer correspondence $\Irr(W_H) \hookrightarrow \mathcal{U}_H^e$. This was completly established by Lusztig in \cite{Lusztig:1984aa} and called the generalized Springer correspondence, as we briefly recall here.

In order to describe the missing pieces, Lusztig defined fundamental blocks, called cuspidal triples, consisting of $H$-conjugacy classes of triples $(L,\mathcal{C}_{v}^L,\varepsilon)$ where $L$ is a Levi subgroup of $H$, $v$ is a unipotent element of $L$ and $\varepsilon \in \Irr(A_L(v))$ is an irreducible \textbf{cuspidal} representation of $A_L(v)$. To each $(\mathcal{C}_u^H,\eta) \in \mathcal{U}_H^e$, he associated a unique triple $(L,\mathcal{C}_v^L,\varepsilon)$; see \cite[6.3,6.4]{Lusztig:1984aa}. All elements associated to a fixed triple $(L,\mathcal{C}_v^L,\varepsilon)$ are parametrized by $\Irr(W_H^L)$ with $W_H^L=N_H(L)/L$ \cite[6.4]{Lusztig:1984aa}.

The notion of cuspidal representation of $A_H(u)$ was introduced by Lusztig in \cite[2.4,6.2]{Lusztig:1984aa} and involves geometric objects. We now review this notion as it appears in \cite{Lusztig:1984aa}. Let $u\in H$ be a unipotent element and $\varepsilon \in \Irr(A_H(u))$.
Let $P=MN$ be a parabolic subgroup of $H$ and $v \in M$ be a unipotent element. Set $$Y_{P,u,v}=\left\{ hZ_M(v)^{\circ}N \in H/Z_M(v)^{\circ}N \mid h \in H, h^{-1}uh \in vN \right\}$$ and $$d_{u,v}=\frac{1}{2}(\dim Z_H(u) - \dim Z_M(v)).$$ 
Then $\dim Y_{P,u,v} \leqslant d_{u,v}$ \cite[1.1]{Lusztig:1984aa}.
The group $Z_H(u)$ acts on $Y_{P,u,v}$ by left translation; this action factorizes to an action of $A_H(u)$ on the set of irreducible components of $Y_{P,u,v}$ of dimension $d_{u,v}$.
Let $S_{u,v}$ be the resulting representation of $A_H(u)$.

Then $\varepsilon$ is a \emph{cuspidal representation of $A_H(u)$} if for all proper parabolic subgroups $P=MN$ of $H$ and for all unipotent $v \in M$, we have 
\[
\Hom_{A_H(u)}(\varepsilon,S_{u,v})=0.
\]

\begin{example}
If $P=B=TU$, then $$Y_{B,u,1}=\left\{ gB \in H/B \mid g \in H, g^{-1}ug \in U \right\}=\left\{B' \in \mathcal{B} \mid u \in B'\right\}=\mathcal{B}_u,$$ which is the Springer fiber of $u$. It was through this variety that Springer established his original correspondence.
\end{example}

We may now state the generalized Springer correspondence.
Let $\mathcal{S}_H$ be the set of $H$-conjugacy classes of cuspidal triples $(L,\mathcal{C}_v^L,\varepsilon)$ where
\begin{itemize}
\item $L$ is a Levi subgroup of $H$;
\item $v \in L$ is a unipotent element of $L$;
\item $\varepsilon \in \Irr(A_L(v))$ is a cuspidal representation.
\end{itemize} 
\begin{theorem}[{Lusztig, \cite[6.5,9.2]{Lusztig:1984aa}}]
Let $H$ be a connected complex algebraic group.
There is a surjective map 
\[
\Psi_H : \mathcal{U}_H^e \longrightarrow \mathcal{S}_H
\]
and, for each $\mathfrak{t}=[L,\mathcal{C}_v^L,\varepsilon]\in \mathcal{S}_H$, a natural bijection 
\[
\Psi_H^{-1}(\mathfrak{t}) \longleftrightarrow \Irr(N_H(L)/L).
\]
\end{theorem}

For all $\mathfrak{t}=[L,\mathcal{C}_v^L,\varepsilon]$, set $\mathcal{M}_{\mathfrak{t}}=\Psi_H^{-1}(\mathfrak{t})$. 
The map $\Psi_H$ induces a decomposition of $\mathcal{U}_H^e$ : $$\mathcal{U}_H^e=\bigsqcup_{\mathfrak{t}\in \mathcal{S}_H} \mathcal{M}_{\mathfrak{t}}.$$ 
The ordinary Springer correspondence is recovered from the Springer correspondence by restricting to the case $\mathfrak{t}=(T,\{1\},1)$ where $T$ is a maximal torus of $H$. It is remarkable that the Levi subgroups $L$ of $H$ which appear in the generalized Springer correspondence for $H$ are very special. In particular, the relative Weyl group $W_L^H=N_H(L)/L$ is a Coxeter group \cite[9.2]{Lusztig:1984aa} which is not true in general. This property is an important fact.

Let us describe the triples $(H,\mathcal{C}_v^H,\varepsilon) \in \mathcal{S}_H$ for certain groups $H$.
\begin{itemize}
\item $H=\SL_n(\C)$. If $(H,\mathcal{C}_v^H,\varepsilon) \in \mathcal{S}_H$ then the unipotent element $v$ corresponds to the partition $(n)$, in which case $A_H(v)=\Z/n\Z$. 
The cuspidal representations $\varepsilon$ which appear in $(H,\mathcal{C}_v^H,\varepsilon) \in \mathcal{S}_H$ are precisely those representations of $A_H(v)$ for which $\ker(\varepsilon)=\{0\}$. In particular, the cardinality of the set of the cuspidal representations of $A_H(v)$ is $\phi(n)$ (Euler's $\phi$-function).
\item $H=\GL_n(\C)$. If $(H,\mathcal{C}_v^H,\varepsilon) \in \mathcal{S}_H$ then necessarily $n=1$, $v=1$ and $A_H(v)=\{1\}$. 
\item $H=\Sp_{2n}(\C)$. If $(H,\mathcal{C}_v^H,\varepsilon) \in \mathcal{S}_H$ then $n=\frac{d(d+1)}{2}$ for some integer $d$ and $v$ corresponds to the partition $(2d,2d-2,\ldots,4,2)$, in which case $A_H(v)=\prod_{i=1}^{d} \langle z_{2i} \rangle \simeq (\Z/2\Z)^d$. The representation $\varepsilon$ which appears in $(H,\mathcal{C}_v^H,\varepsilon) \in \mathcal{S}_H$ is precisely that for which $\varepsilon(z_{2i})=(-1)^i$.
\item $H=\SO_{n}(\C)$. If $(H,\mathcal{C}_v^H,\varepsilon) \in \mathcal{S}_H$ then $n=d^2$ for some integer $d$ and $v$ corresponds to the partition $(2d-1,2d-3,\ldots,3,1)$, in which case $A_H(v)=\prod_{i=1}^{d-1} \langle z_{2i+1}z_{2i-1} \rangle \simeq (\Z/2\Z)^{d-1}$ and $\varepsilon(z_{2i+1}z_{2i-1})=-1$.
\end{itemize}

\begin{example}
We come back to our example of $\Sp_6(\C)$. The Levi subgroups of $H=\Sp_6(\C)$ are: $\Sp_6(\C), \;\; \GL_1(\C) \times \Sp_4(\C), \;\; \GL_2(\C) \times \Sp_2(\C), \;\; \GL_1(\C)^2 \times \Sp_2(\C), \;\; \GL_3(\C), \;\; \GL_2(\C) \times \GL_1(\C), \;\;\GL_1(\C)^3.$
The only Levi subgroups of $\Sp_6(\C)$ which can appear in the Springer correspondence for $\Sp_6(\C)$ are $H=\Sp_6(\C), M=\GL_1(\C)^2 \times \Sp_2(\C)$ and $T=\GL_1(\C)^3$. We have : 
\begin{table}[H]
\begin{center}
\renewcommand\arraystretch{1.2}
$$\begin{array}{|c|c|}
\hhline{--}
L & N_H(L)/L \\ \hhline{==}
\Sp_6(\C) & \{1\} \\ \hline
\GL_1(\C)^2 \times \Sp_2(\C) & \mathfrak{S}_2 \ltimes (\Z/2\Z)^2  \\ \hhline{--}
\GL_1(\C)^3 & \mathfrak{S}_3 \ltimes (\Z/2\Z)^3 \\ \hhline{--}
\end{array}$$
\end{center}
\end{table}
\end{example}
Table~\ref{table:Sp6-2} describes the generalized Springer correspondence for $\Sp_6(\C)$; the meaning of the $u$-symbols is given in \cite[3.2]{Achar:2007} or \cite[13.3]{Carter:1993} and included here only for completeness.

\begin{table}[H]
\begin{center}
\renewcommand\arraystretch{1.2}
$$\begin{array}{|c|c|c|c|c|c|}
\hhline{------}
u & A_H(u) & \Irr(A_H(u)) & \text{$u$ symbol}& L & \Irr(W_H^L) \\ \hhline{======}
 \multirow{2}*{$(6)$} & \multirow{2}*{$\Z/2\Z$} & \cellcolor{gris1} 1 &\cellcolor{gris1} \left( \begin{smallmatrix} 3 \\ - \end{smallmatrix} \right)  & \cellcolor{gris1} T & \cellcolor{gris1} \rho_{(3,\varnothing)} \\  \hhline{~~----}
 &  & \cellcolor{gris2} \zeta & \cellcolor{gris2}\left( \begin{smallmatrix}  - \\  3 \end{smallmatrix} \right)   & \cellcolor{gris2}M & \cellcolor{gris2} \rho_{(2,\varnothing)}'\\ \hhline{------}
\multirow{4}*{$ (4,2)$} & \multirow{4}*{$(\Z/2\Z)^2$} &\cellcolor{gris1} 1 \boxtimes 1 &\cellcolor{gris1} \left( \begin{smallmatrix} 0 &  & 4 \\   & 2 & \end{smallmatrix} \right)  & \cellcolor{gris1} T & \cellcolor{gris1} \rho_{(2,1)} \\ \hhline{~~----}
 & & \cellcolor{gris1}\zeta \boxtimes \zeta &\cellcolor{gris1} \left( \begin{smallmatrix} 0 &  & 2 \\   & 4 & \end{smallmatrix} \right)   &\cellcolor{gris1} T &\cellcolor{gris1} \rho_{(\varnothing,3)}  \\ \hhline{~~----}
 &  &\cellcolor{gris2} 1 \boxtimes \zeta &\cellcolor{gris2} \left( \begin{smallmatrix}  & 0 &  \\ 2 & • & 4\end{smallmatrix} \right)   & \cellcolor{gris2} M & \cellcolor{gris2} \rho_{(1^2,\varnothing)}'\\\hhline{~~----}
&  &\cellcolor{gris3} \zeta \boxtimes 1 &\cellcolor{gris3} \left( \begin{smallmatrix} 0 & 2 & 4 \\  & - & \end{smallmatrix} \right)   & \cellcolor{gris3} H & \cellcolor{gris3} 1 \\ \hhline{------}
\multirow{2}*{$ (4,1^2)$} & \multirow{2}*{$\Z/2\Z$} & \cellcolor{gris1} 1 & \cellcolor{gris1} \left( \begin{smallmatrix} 1 &  & 4 \\  & 1 & \end{smallmatrix} \right)  &\cellcolor{gris1} T & \cellcolor{gris1} \rho_{((2,1),\varnothing)} \\ \hhline{~~----}
 &  &\cellcolor{gris2} \zeta &\cellcolor{gris2} \left( \begin{smallmatrix}  & 1 &  \\  1&  & 4 \end{smallmatrix} \right) & \cellcolor{gris2} M & \cellcolor{gris2}\rho_{(1,1)}'\\ \hhline{------}
 (3^2) & \{1\} &\cellcolor{gris1} 1 &\cellcolor{gris1} \left( \begin{smallmatrix} 0 &  & 3 \\  & 3 & \end{smallmatrix} \right)   &\cellcolor{gris1} T &\cellcolor{gris1} \rho_{(1,2)} \\ \hhline{------}
\multirow{2}*{$(2^3)$} & \multirow{2}*{$\Z/2\Z$} &\cellcolor{gris1} 1&\cellcolor{gris1} \left( \begin{smallmatrix} 1 &  & 3 \\  & 2 & \end{smallmatrix} \right)  & \cellcolor{gris1}T &\cellcolor{gris1}\rho_{(1^2,1)}  \\ \hhline{~~----}
& & \cellcolor{gris2}\zeta &\cellcolor{gris2} \left( \begin{smallmatrix}  & 2 &  \\ 1 &  & 3 \end{smallmatrix} \right)   & \cellcolor{gris2} M &\cellcolor{gris2} \rho_{(\varnothing,2)}'\\ \hhline{------}
\multirow{2}*{$ (2^2,1^2)$} & \multirow{2}*{$\Z/2\Z$} & \cellcolor{gris1} 1 &\cellcolor{gris1}  \left( \begin{smallmatrix} 0 &  & 2 &  & 5 \\  & 2 &  & 4 &  \end{smallmatrix} \right)  &\cellcolor{gris1} T &\cellcolor{gris1} \rho_{(1,1^2)} \\ \hhline{~~----}
 & & \cellcolor{gris1} \zeta & \cellcolor{gris1} \left( \begin{smallmatrix} 0 &  & 2 &  & 4 \\  & 2 &  & 5 &  \end{smallmatrix} \right)  &\cellcolor{gris1} T &\cellcolor{gris1} \rho_{(\varnothing,(2,1))} \\ \hhline{------}
\multirow{2}*{$(2,1^4)$} & \multirow{2}*{$\Z/2\Z$} &\cellcolor{gris1} 1 &\cellcolor{gris1}  \left( \begin{smallmatrix} 1 &  & 3 &  & 5 \\  & 1 &  & 3 &  \end{smallmatrix} \right) & \cellcolor{gris1} T & \cellcolor{gris1} \rho_{(1^3,\varnothing)}  \\ \hhline{~~----}
&  & \cellcolor{gris2} \zeta &\cellcolor{gris2} \left( \begin{smallmatrix} & 1 &  & 3 &  \\  1 &  & 3 &  & 5 \end{smallmatrix} \right) & \cellcolor{gris2}M & \cellcolor{gris2}\rho_{(\varnothing,1^2)}'\\ \hhline{------}
(1^6) & \{1\} & \cellcolor{gris1} 1 & \cellcolor{gris1} \left( \begin{smallmatrix}0 &  & 2 &  & 4 &  & 6 \\  & 2 &  & 4 &  & 6 & \end{smallmatrix} \right)   & \cellcolor{gris1} T & \cellcolor{gris1} \rho_{(\varnothing,1^3)} \\ \hhline{------}
\end{array} $$
\end{center}
\caption{Generalized Springer correspondence for $\Sp_6(\C)$}
\label{table:Sp6-2}
\end{table}

Table~\ref{table:SO4} describes the generalized Springer correspondence for $\SO_4(\C)$; we will use this example in the Example \ref{exampleABPS}. The partition $(2^2)$ parametrizes two distinct unipotent classes of $\SO_4(\C)$; we distinguished them by adding an apostrophe. The irreducible representations of the Weyl group of type $D_{n}$ are described by the sets $\{\alpha,\beta\}$ where $\alpha$ and $\beta$ are partitions (perhaps trivial) such that $|\alpha| + |\beta|=n$. When $n$ is even and $\alpha=\beta$ then there are two distinct representations parametrized by $\{\alpha,\alpha\}$ and we distinguished them by adding an apostrophe.

\begin{table}[H]
\begin{center}
\renewcommand\arraystretch{1.2}
$$\begin{array}{|c|c|c|c|c|c|c|}
\hhline{-------}
u & A_H(u) & \Irr(A_H(u)) & \text{$u$ symbol}& L & \Irr(W_H^L) & \Irr(W_H^L) \otimes \sgn \\ \hhline{=======}
 \multirow{2}*{$(3,1)$} & \multirow{2}*{$\Z/2\Z$} & \cellcolor{gris1} 1 &\cellcolor{gris1} \left( \begin{smallmatrix} 0 \\ 2 \end{smallmatrix} \right)  & \cellcolor{gris1} T & \cellcolor{gris1} \rho_{\{2,\varnothing\}} & \cellcolor{gris1} \rho_{\{1^2,\varnothing\}} \\  \hhline{~~-----}
 &  & \cellcolor{gris2} \zeta & \cellcolor{gris2}\left( \begin{smallmatrix}  0 & & 2 \\  & - &  \end{smallmatrix} \right)   & \cellcolor{gris2}H & \cellcolor{gris2} 1 & \cellcolor{gris2} 1 \\ \hhline{-------}
(2^2) & \{1\} & \cellcolor{gris1} 1 & \cellcolor{gris1}\left( \begin{smallmatrix}  1 \\ 1  \end{smallmatrix} \right)   & \cellcolor{gris1} T & \cellcolor{gris1} \rho_{\{1,1\}} & \cellcolor{gris1} \rho_{\{1,1\}}' \\ \hhline{-------}
(2^2)' & \{1\} & \cellcolor{gris1} 1 & \cellcolor{gris1}\left( \begin{smallmatrix}  1 \\ 1  \end{smallmatrix} \right)   & \cellcolor{gris1} T & \cellcolor{gris1} \rho_{\{1,1\}}' & \cellcolor{gris1} \rho_{\{1,1\}} \\ \hhline{-------}
(1^4) & \{1\} & \cellcolor{gris1} 1 & \cellcolor{gris1}\left( \begin{smallmatrix}  0 & 2  \\ 1 & 3 \end{smallmatrix} \right)   & \cellcolor{gris1} T & \cellcolor{gris1} \rho_{\{1^2,\varnothing\}} & \cellcolor{gris1} \rho_{\{2,\varnothing\}}  \\ \hhline{-------}
\end{array} $$
\end{center}
\caption{Generalized Springer correspondence for $\SO_4(\C)$}
\label{table:SO4}
\end{table}

\subsection{Generalized Springer correspondence for orthogonal groups}

Let $n \geqslant1$ be an integer. In this paragraph $H$ denotes the orthogonal group $\O_n(\C)$ or the group  $\left\{ (x_i) \in \prod_{i=1}^{m} \O_{n_i}(\C) \mid \prod_{i=1}^{m} \det(x_i)=1 \right\}$ . Note that $H$ is disconnected. Here we specialize the definitions appearing in Section~\ref{sec:GSC} to this case and also state the generalized Springer correspondence for in this case. First, recall that $$\mathcal{U}_H^{e}=\left\{ (\mathcal{C}_{u}^{H}, \eta) \mid u \in H \mbox{ unipotent}, \eta \in \Irr(A_H(u)) \right\}.$$

\begin{definition}[{\cite[A.1]{Moussaoui:2015}}]
A subgroup $L$ of $H$ is said to be a \emph{quasi-Levi subgroup} of $H$ if there exists a torus $A \subset H^{\circ}$ such that $L=Z_H(A)$.
\end{definition}

\begin{example}

\begin{table}[H]
\begin{center}
\renewcommand{\arraystretch}{1.2}
$$\begin{array}{|c|c|c|c|c|c}
\hhline{-----~}
H & L^{\circ} & L & L/L^{\circ} & W_{L}^{H}/W_{L^{\circ}}^{H^{\circ}} &\\ \hhline{=====~}
\O_{2n+1} & \prod_{i=1}^{k} \GL_{n_i} \times \SO_{2n'+1} & \prod_{i=1}^{k} \GL_{n_i} \times \O_{2n'+1} & \Z/2\Z & \{1\} &  n_i \geqslant 0, n' \geqslant 0\\ 
\O_{2n}& \prod_{i=1}^{k} \GL_{n_i} \times \SO_{2n'} & \prod_{i=1}^{k} \GL_{n_i} \times \O_{2n'} & \Z/2\Z & \{1\} & n_i \geqslant 0, n' \geqslant 2 \\ 
 & \prod_{i=1}^{k} \GL_{n_i} & \prod_{i=1}^{k} \GL_{n_i} & \{ 1 \} & \Z/2\Z & n_i \geqslant 0\\ \hhline{-----~}
\end{array} $$
\end{center}
\caption{Levi and quasi-Levi subgroups of orthogonal groups}\label{exemplequasiLevi}
\end{table}
\end{example}

In the previous definition, since $L^{\circ}=Z_{H^{\circ}}(A)$, then $L^{\circ}$ is Levi subgroup of $H^{\circ}$. If $L$ is a quasi-Levi subgroup of $H$ and $u \in H^{\circ}$ is unipotent, then $W_{H^{\circ}}^{L^{\circ}}$ and $A_{H^{\circ}}(u)$ are normal subgroups of $W_{H}^{L}=N_{H}(L)/L$ and $A_H(u)$, respectively. If $\varepsilon \in \Irr(A_H(u))$, the restriction to the normal subgroup $A_{H^{\circ}}(u)$ decomposes as $$\restriction{\varepsilon}{A_{H^{\circ}}(u)}=\bigoplus_{\tau} \tau \otimes \C^{m},$$ where $\tau$ runs over some irreducible representations of $A_{H^{\circ}}(u)$ which are all conjugate by $A_H(u)$ and $m \geqslant 1$ is an integer.

\begin{definition}[{\cite[A.3]{Moussaoui:2015}}]\label{defcuspA}
Let $\varepsilon \in \Irr(A_H(u))$. Then $\varepsilon$ is a \emph{cuspidal representation of $A_{H}(u)$} if the irreducible representations of $A_{H^{\circ}}(u)$ which appear in the  restriction of $\varepsilon$ to $A_{H^{\circ}}(u)$ are cuspidal in the sense of Lusztig (recalled in Section~\ref{sec:GSC}.)
\end{definition}

Notice that in the restriction of $\varepsilon$ to $A_{H^{\circ}}(u)$, all the representations of $A_{H^{\circ}}(u)$ which appears are conjugate under $A_H(u)$. In particular, the cuspidality is preserved by such conjugation. As a consequence, one representation in the restriction is cuspidal if and only if all the representations are cuspidal.

We may now state the generalized Springer correspondence for orthogonal groups and some subgroups of orthogonal groups.
Let $\mathcal{S}_H$ be the set of $H$-conjugacy classes of triples $(L,\mathcal{C}_v^L,\varepsilon)$ for which 
\begin{itemize}
\item $L$ is a quasi-Levi subgroup of $H$ ;
\item $v \in L^{\circ}$ is a unipotent element ;
\item $\varepsilon \in \Irr(A_L(v))$ is a cuspidal representation.
\end{itemize}

\begin{theorem}\cite[A4,A8]{Moussaoui:2015}
Assume $H=\left\{ (x_i) \in \prod_{i=1}^{m} \O_{n_i}(\C) \mid \prod_{i=1}^{m} \det(x_i)=1 \right\}$ or $H=\O_n(\C)$.
 There is a surjective map 
\[
\Psi_H : \mathcal{U}_H^e \longrightarrow \mathcal{S}_H
\]
and, for each $\mathfrak{t}=[L,\mathcal{C}_v^L,\varepsilon]\in \mathcal{S}_H$, a natural bijection 
\[
\Psi_H^{-1}(\mathfrak{t}) \longleftrightarrow \Irr(N_H(L)/L).
\]
\end{theorem}

For all $\mathfrak{t}=[L,\mathcal{C}_v^L,\varepsilon]$, set $\mathcal{M}_{\mathfrak{t}}=\Psi_H^{-1}(\mathfrak{t})$. 
The map $\Psi_H$ induces a decomposition of $\mathcal{U}_H^e$ : $$\mathcal{U}_H^e=\bigsqcup_{\mathfrak{t}\in \mathcal{S}_H} \mathcal{M}_{\mathfrak{t}}.$$ 

\section{Relation between the local Langlands correspondence and the Bernstein decomposition}\label{sec:Langlands}

\subsection{Local Langlands correspondence}
Let $F$ be a finite extension of a $p$-adic field $\Q_p$ and $G$ be (the $F$-points of) a split connected reductive group over $F$. Let $W_F$ (resp. $W_F'=W_F \times \SL_2(\C)$) be the Weil (resp. Weil-Deligne) group of $F$. We denote by $\widehat{G}$ the connected complex reductive group dual of $G$. A Langlands parameter for $G$ is a group homorphism $$\phi : W_F' \longrightarrow \widehat{G},$$ such that: 
\begin{itemize}
\item the restriction to $\SL_2(\C)$ is a morphism of algebraic groups;
\item the restriction to $W_F$ is continuous and $\phi(W_F)$ consists of semi-simple elements.\end{itemize} If we denote by $\Phi(G)$ the set of $\widehat{G}$-conjugacy classes of Langlands parameters for $G$. The local Langlands correspondence predicts the existence of a finite-to-one map $$\rec_G : \Irr(G) \longrightarrow \Phi(G),$$ which satisfies certain properties. To each $\phi \in \Phi(G)$, one can expect to attach a $L$-packet $\Pi_{\phi}(G)$ which is a finite set of irreducible representations of $G$ associated to $\phi$. Conjecturally, this set is parametrized by the irreducible representations of a finite group $\mathcal{S}_{\phi}^{G}$ which is a quotient of $A_{\widehat{G}}(\phi)$, the component group of the centralizer in $\widehat{G}$ of $\phi(W_F')$. Hence, if we denote by $$\Phi_{e}(G)=\{(\phi,\eta) \mid \phi \in \Phi(G), \eta \in \Irr(\mathcal{S}_{\phi}^{G})\},$$ the set of \emph{enhanced} Langlands parameters, then conjecturally, we have a bijection $$\Irr(G) \longleftrightarrow \Phi_{e}(G),$$ and a decomposition $$\Irr(G) = \bigsqcup_{\phi \in \Phi(G)} \Pi_{\phi}(G).$$

\subsection{Stable Bernstein centre}

Recently, inspired by Vogan, in \cite{Haines:2013aa} Haines has defined the stable Bernstein centre and stated some conjectures and gave some properties. In this paper we only consider the split case, but Haines treats the general case. One can view the stable Bernstein centre as an analogue of the Bernstein centre but for the Langlands parameters. 
It is conjectured that the Langlands correspondence is compatible with parabolic induction (see conjecture \ref{conjindLLC}). 
Haines defines a cuspidal datum as a pair $(\widehat{M}, \lambda)$ with $\widehat{M}$ a Levi subgroup of $\widehat{G}$ and $\lambda : W_F \longrightarrow \widehat{M}$ a discrete Langlands parameter for $M$ (which means that the image of the parameter does not factorize through a proper Levi subgroup of $\widehat{M}$). 
This plays the role of cuspidal data for Langlands parameters. 
Also, he attaches to each Langlands parameter of $G$ a cuspidal datum and an inertial class. If $\widehat{M}$ is a Levi subgroup of $\widehat{G}$, we denote by $\mathcal{X}(\widehat{M})=\left\{ \chi : W_F/I_F \longrightarrow Z_{\widehat{M}}^{\circ} \right\}$. Then by the Langlands correspondence for characters, $\mathcal{X}(\widehat{M})$ is in bijection with the group $\mathcal{X}(M)$ of the unramified characters of $M$.
Following Haines \cite[5.3.3]{Haines:2013aa}, consider two equivalence relations $\sim_{\Omega}$ and $\sim_{\mathcal{B}}$ on the pairs $(\widehat{M},\lambda)$ with $\widehat{M}$ a Levi subgroup of $\widehat{G}$ and $\lambda : W_F \longrightarrow \widehat{M}$ a discrete Langlands parameter of $M$ (trivial on $\SL_2(\C)$) defined by: 
\begin{enumerate}
\item $(\widehat{M}_1,\lambda_1) \sim_{\Omega} (\widehat{M}_2,\lambda_2)$ if and only if there exists $g \in \widehat{G}$ such that ${}^g \widehat{M}_1=\widehat{M}_2$ and ${}^g \lambda_1=\lambda_2$;
\item $(\widehat{M}_1,\lambda_1) \sim_{\mathcal{B}} (\widehat{M}_2,\lambda_2)$ if and only if there exists $g \in \widehat{G}$ and $\chi \in \mathcal{X}(\widehat{M}_2)$ such that ${}^g \widehat{M}_1=\widehat{M}_2$ and ${}^g \lambda_1=\lambda_2 \chi_2$.
\end{enumerate}
We denote by $(\lambda)_{\widehat{M}}$ the $\widehat{M}$-conjugacy class of $\lambda$. Moreover, if we denote by $\Omega^{\st}(G)$ (resp. $\mathcal{B}^{\st}(G)$) the equivalence classes for the relation $\sim_{\Omega}$ (resp. $\sim_{\mathcal{B}}$) then $$\Omega^{\st}(G)=\bigsqcup_{\wi \in \mathcal{B}^{\st}(G)} \mathcal{T}_{\wi} /\mathcal{W}_{\wi},$$ with if $\wi=[\widehat{M},\lambda]$ : \begin{itemize}
\item $\mathcal{T}_{\wi}=\{( \lambda \chi)_{\widehat{M}}, \chi \in  \mathcal{X}(\widehat{M}) \} \simeq  \mathcal{X}(\widehat{M})/\mathcal{X}(\widehat{M})(\lambda)$ and $\mathcal{X}(\widehat{M})(\lambda)=\{ \chi \in  \mathcal{X}(\widehat{M}) \mid (\lambda)_{\widehat{M}}=(\lambda \chi)_{\widehat{M}}\}$ ;
\item $\mathcal{W}_{\wi} =\{w \in N_{\widehat{G}}(\widehat{M})/\widehat{M} \mid \exists \chi \in \mathcal{X}(\widehat{M}), ({}^w \lambda)_{\widehat{M}}=(\lambda\chi)_{\widehat{M}} \}$
\end{itemize} 
To each Langlands parameter $\phi$, one can define its infinitesimal character $\lambda_{\phi}$ by defining for all $w \in W_F$ $$\lambda_{\phi}(w)=\phi(w,d_w),$$ with $d_w=\diag(|w|^{1/2}, |w|^{-1/2})$.

\begin{definition}[{Haines \cite[5.1]{Haines:2013aa}}]
Let $\widehat{M}$ be a Levi subgroup of $\widehat{G}$ and $\lambda : W_F \longrightarrow \widehat{M} \hookrightarrow \widehat{G}$ be a discrete Langlands parameter of $M$. Let $\wi=[\widehat{M},\lambda] \in \mathcal{B}^{\st}(G)$ the inertial class defined by $(\widehat{M},\lambda)$. Then the infinitesimal packet of $\lambda$ is
$$\Pi_{\lambda}^{+}(G)=\bigsqcup_{\substack{\phi \in \Phi(G) \\ \lambda_{\phi}=\lambda}} \Pi_{\phi}(G),$$ and the inertial packet of $\lambda$ is $$\Pi_{\wi}^{+}(G)=\bigsqcup_{\substack{\phi \in \Phi(G) \\ \lambda_{\phi}=\lambda\chi \\ \chi \in  \mathcal{X}(\widehat{M})}} \Pi_{\phi}(G)=\bigsqcup_{\lambda\chi \in \mathcal{T}_{\wi}/\mathcal{W}_{\wi}} \Pi_{\lambda\chi}^{+}(G).$$
\end{definition}

\begin{conj}[{Haines \cite[5.2.2]{Haines:2013aa},Vogan}]\label{conjindLLC}
Let $\sigma$ be an irreducible supercuspidal representation of $M$ and $\pi$ be an irreducible subquotient of the parabolically induced representation $i_{P}^G (\sigma)$ (where $P$ is a parabolic subgroup with Levi factor $M$). By the Langlands correspondence, let $$\phi_{\sigma}: W_F' \longrightarrow \widehat{M},$$ and $$\phi_{\pi} : W_F' \longrightarrow \widehat{G},$$ be the respective Langlands parameters of $\sigma$ and $\pi$. By the embedding $\widehat{M} \hookrightarrow \widehat{G}$, one can view $\phi_{\sigma}$ as a Langlands parameter of $G$. Then it is expected that we have the following equality : $$(\lambda_{\phi_\sigma})_{\widehat{G}}=(\lambda_{\phi_\pi})_{\widehat{G}}.$$
\end{conj}

Currently, this conjecture is proved for $\GL_n$ (essentially from the proof of the Langlands correspondence see \cite[5.2.3]{Haines:2013aa}) and for the split classical groups by \cite[4.7]{Moussaoui:2015}.

\begin{example}
Let $G=\GL_2(F)$, $T \simeq (F^{\times})^2$ be the maximal torus of $G$ consisting of diagonal matrices. Let $| \, \cdot \,| $ be the norm of $F^{\times}$. Consider the irreducible supercuspidal representation $\sigma= | \, \cdot \,|^{1/2} \boxtimes | \, \cdot \,|^{-1/2}$ of $T$. Then the induced representation $i_B^G(\sigma)$ has two irreducibles subquotients : $\pi_1=1_{\GL_2}$ the trivial representation of $G$ and $\pi_2=\St_{\GL_2}$ the Steinberg representation of $G$.\\

The Langlands parameters of $\sigma, \pi_1$ and $\pi_2$ are, respectively:
$$\begin{array}[t]{cccl}
\phi_{\sigma} :  & W_F' & \longrightarrow & \widehat{T} \\
&(w,x) & \longmapsto & \diag(|w|^{1/2}, |w|^{-1/2})\\
\phi_{\pi_1} : & W_F' & \longrightarrow & \widehat{G} \\
&(w,x) & \longmapsto & \diag(|w|^{1/2}, |w|^{-1/2})\\
\phi_{\pi_2} : & W_F' & \longrightarrow & \widehat{G} \\
&(w,x) & \longmapsto & x
 \end{array}$$
Hence, we have $\lambda_{\phi_{\sigma}}=\lambda_{\phi_{\pi_1}}=\lambda_{\phi_{\pi_2}}$.
\end{example}

\subsection{Cuspidal enhanced Langlands parameter}\label{sectioncuspidallanglandsparameter}

Recall that we have two decompositions of $\Irr(G)$, one by the Bernstein decomposition, the other by the Langlands correspondence: $$\Irr(G)=\bigsqcup_{\mathfrak{s} \in \mathcal{B}(G)} \Irr(G)_{\mathfrak{s}} = \bigsqcup_{\phi \in \Phi(G)} \Pi_{\phi}(G).$$ We want to compare the two decompositions, in particular, we want to describe the Langlands parameters of supercuspidal representations and the cuspidal support map.\\

If $\varphi \in \Phi(G)$, recall that we have two groups $A_{\widehat{G}}(\varphi)$ and $\mathcal{S}_{\varphi}^G$ defined by : $$A_{\widehat{G}}(\varphi)=Z_{\widehat{G}}(\varphi)/Z_{\widehat{G}}(\varphi)^{\circ} \mbox{  and  } \mathcal{S}_{\varphi}^G=Z_{\widehat{G}}(\varphi)/Z_{\widehat{G}}(\varphi)^{\circ} \cdot Z_{\widehat{G}}.$$
Conjecturally $\Irr(\mathcal{S}_{\varphi}^G)$ parametrizes the $L$-packet $\Pi_{\varphi}(G)$ and we have a surjective map $A_{\widehat{G}}(\varphi) \twoheadrightarrow \mathcal{S}_{\varphi}^G$. We remark that if we denote  $H_{\varphi}^{G}=Z_{\widehat{G}}(\restriction{\varphi}{W_F})$, then we have the following equalities $$Z_{\widehat{G}}(\varphi)=Z_{\widehat{G}}(\restriction{\varphi}{W_F}) \cap Z_{\widehat{G}}(\restriction{\varphi}{\SL_2})=Z_{Z_{\widehat{G}}(\restriction{\varphi}{W_F})}(\restriction{\varphi}{\SL_2})=Z_{H_{\varphi}^{G}}(\restriction{\varphi}{\SL_2}).$$ The group $H_{\varphi}^{G}$ is a reductive group and if $u_{\varphi}=\varphi \left( \begin{pmatrix} 1 & 1 \\  0 & 1 \end{pmatrix} \right)$, then $A_{\widehat{G}}(\varphi)=A_{H_{\varphi}^{G}}(\restriction{\varphi}{\SL_2})$ and $A_{H_{\varphi}^{G}}(\restriction{\varphi}{\SL_2})=A_{H_{\varphi}^{G}}(u_{\varphi})$. In general, $H_{\varphi}^G$ is a disconnected group. 

\begin{definition}[{\cite[3.4]{Moussaoui:2015}}]\label{defcuspS}
Let $\varphi \in \Phi(G)$ be a discrete Langlands parameter, $\varepsilon \in \Irr(\mathcal{S}_{\varphi}^G)$ and $\widetilde{\varepsilon}$ the pullback of $\varepsilon$ to $A_{\widehat{G}}(\varphi)=A_{H_{\varphi}^G}(u_{\varphi})$. One says that $\varepsilon$ is a \emph{cuspidal representation of $\mathcal{S}_{\varphi}^G$} when $\widetilde{\varepsilon}$ is cuspidal with respect to the group $H_{\varphi}^G$ and $u_{\varphi}$ (see Definition~\ref{defcuspA}). We denote by $\Irr(\mathcal{S}_{\varphi}^G)_{\cusp}$ the set of irreducible cuspidal representations of $\mathcal{S}_{\varphi}^G$.
Moreover, one says that $\varphi$ is a \emph{cuspidal parameter} when $\Irr(\mathcal{S}_{\varphi}^G)_{\cusp}$ is not empty.
\end{definition}

\begin{conj}[{\cite[3.5]{Moussaoui:2015}}]\label{conjecturecusp}
Let $\varphi \in \Phi(G)$ be a Langlands parameter of $G$. The $L$-packet $\Pi_{\varphi}(G)$ contains supercuspidal representations of $G$ if and only if $\varphi $ is a cuspidal parameter of $G$. Moreover, if $\varphi$ is a cuspidal parameter of $G$, the supercuspidal representations in $\Pi_{\varphi}(G)$ are parametrized by $\Irr(\mathcal{S}_{\varphi}^{G})_{\cusp}$; in other words, there is a bijection $$\Pi_{\varphi}(G)_{\cusp} \longleftrightarrow \Irr(\mathcal{S}_{\varphi}^{G})_{\cusp}.$$
\end{conj}

In the following we describe the cuspidal Langlands parameters for $\GL_n(F), \,\, \Sp_{2n}(F)$ and $\SO_n(F)$. We denote by $S_a$ the irreducible representation of dimension $a$ of $\SL_2(\C)$ and by $I_O$ (resp. $I_S$) a set of irreducible representations of $W_F$ of orthogonal type (resp. symplectic type). This means for $I_O$ (resp. $I_S$) that the image of $\pi \in I_O$ can be factorized through an orthogonal group (resp. symplectic group).

\begin{prop}[{\cite[3.7]{Moussaoui:2015}}]\label{proparamcusp}
We keep same notations as before. The cuspidal Langlands parameters for $G$ are :
\begin{itemize}
\item $\GL_n(F)$, $$\varphi : W_F \longrightarrow \GL_n(\C), \,\, \text{irreducible (or equivalently, discrete)} \,\, ;$$
\item $\SO_{2n+1}(F)$, $$\varphi=\bigoplus_{\pi \in I_{O}} \bigoplus_{a=1}^{d_{\pi}} \pi \boxtimes  S_{2a} \bigoplus_{\pi \in I_{S}} \bigoplus_{a=1}^{d_{\pi}} \pi \boxtimes  S_{2a-1}, \,\, \forall \pi \in I_O, d_{\pi} \in \N, \,\, \forall \pi \in I_S, d_{\pi} \in \N^{*} ;$$
\item $\Sp_{2n}(F)$ or $\SO_{2n}(F)$, $$\varphi=\bigoplus_{\pi \in I_{S}} \bigoplus_{a=1}^{d_{\pi}} \pi \boxtimes  S_{2a} \bigoplus_{\pi \in I_{O}} \bigoplus_{a=1}^{d_{\pi}} \pi \boxtimes  S_{2a-1}, \,\, \forall \pi \in I_O, d_{\pi} \in \N^{*}, \,\, \forall \pi \in I_S, d_{\pi} \in \N.$$
\end{itemize}
The conjecture \ref{conjecturecusp} is true for $\GL_n(F), \,\, \Sp_{2n}(F)$ and $\SO_{n}(F)$.
\end{prop}

The last part follows by comparison the work of Harris-Taylor, Henniart or Scholze for $\GL_n(F)$ and the work of Arthur and M\oe glin for the classical groups.

\subsection{Cuspidal support}

The cuspidal support of an irreducible representation of $G$ is a class (of $G$-conjugation) of a pair $(L,\sigma)$ with $L$ a Levi subgroup of $G$ and $\sigma$ an irreducible supercuspidal representation of $L$. By our previous conjecture \ref{conjecturecusp}, each such pair should correspond on the Galois side to a triple $(\widehat{L},\varphi,\varepsilon)$ with $\widehat{L}$ a Levi subgroup of $\widehat{G}$, $(\varphi,\varepsilon) \in \Phi_e(L)_{\cusp}$. \\

Recall that we denote $\mathcal{X}(\widehat{G})=\left\{ \chi : W_F/I_F \longrightarrow Z_{\widehat{G}}^{\circ} \right\}$ and that there is a bijection between $\mathcal{X}(\widehat{G})$ and the unramified characters of $G$. Define two relations  $\sim_{\Omega_e}$ and $\sim_{\mathcal{B}_e}$ on the (set of) triples $(\widehat{L},\varphi,\varepsilon)$ as in the previous paragraph : \begin{enumerate}
\item $(\widehat{L}_1,\varphi_1,\varepsilon_1) \sim_{\Omega_e} (\widehat{L}_2,\varphi_2,\varepsilon_2)$ if and only if there exists $g \in \widehat{G}$ such that ${}^g \widehat{L}_1=\widehat{L}_2, \,\, {}^g \varphi_1=\varphi_2$ and $\varepsilon_1^g=\varepsilon_2$ ;
\item $(\widehat{L}_1,\varphi_1,\varepsilon_1) \sim_{\mathcal{B}_e} (\widehat{L}_2,\varphi_2,\varepsilon_2)$ if and only if there exist $g \in \widehat{G}$ and $\chi \in  \mathcal{X}(\widehat{L}_2)$ such that ${}^g \widehat{L}_1=\widehat{L}_2, \,\, {}^g \varphi_1=\varphi_2 \chi_2$ and $\varepsilon_1^g=\varepsilon_2$ .
\end{enumerate}

Denote by $\Omega_{e}^{\st}(G)$ (resp. by $\mathcal{B}_{e}^{\st}(G)$) the equivalence classes of the relation  $\sim_{\Omega_e}$ (resp. $\sim_{\mathcal{B}_e}$). As before, we have $$\Omega_{e}^{\st}(G)=\bigsqcup_{\wj \in \mathcal{B}_{e}^{\st}(G)} \mathcal{T}_{\wj} /\mathcal{W}_{\wj},$$ with if $\wj=[\widehat{L},\varphi,\varepsilon]$ : \begin{itemize}
\item $\mathcal{T}_{\wj}=\{( \varphi\chi)_{\widehat{L}}, \chi \in \mathcal{X}(\widehat{L}) \} \simeq  \mathcal{X}(\widehat{L})/\mathcal{X}(\widehat{L})(\varphi)$ and $ \mathcal{X}(\widehat{L})(\varphi)=\{\chi \in \mathcal{X}(\widehat{L}) \mid (\varphi)_{\widehat{L}}=(\varphi \chi)_{\widehat{L}}\}$ ; 
\item $\mathcal{W}_{\wj} =\{w \in N_{\widehat{G}}(\widehat{L})/\widehat{L} \mid \exists \chi \in  \mathcal{X}(\widehat{L}), ({}^w \varphi)_{\widehat{L}}=(\varphi\chi)_{\widehat{L}} , \varepsilon^w \simeq \varepsilon \}$
\end{itemize} 

We use the bijection between $\Irr(G)$ and $\Phi_{e}(G)$ given by the local Langlands correspondence; we also use a bijection between $\Omega(G)$ and $\Omega_{e}^{\st}(G)$ found by combining the local Langlands correspondence for supercuspidal representations of the Levi subgroup of $G$ with proposition \ref{proparamcusp} and conjecture \ref{conjecturecusp}.
It follows that there is a cuspidal support map $\Phi_e(G) \to \Omega_e^\st(G)$ such that the following diagram is commutative:
\begin{center}
\begin{tikzcd}
 \Irr(G)   \arrow[swap]{d}{\Sc} \arrow{r}{\rec_G^e} &  \Phi_{e}(G) \arrow[dotted]{d}\\
\Omega(G)   \arrow[swap]{r}{\rec_{\Omega(G)}^{e}}   &\Omega_{e}^{\st}(G)
  \end{tikzcd}
\end{center}
It would be more interesting to define the cuspidal support of $(\phi,\eta)\in \Phi_e(G)$ without assuming the local Langlands correspondence. We solve that problem in the following theorem.

\begin{theorem}[{\cite[3.20]{Moussaoui:2015}}]\label{theoremesupportcuspidal}
Let $G$ be a split classical group, \it{i.e.} $G=\Sp_{2n}(F)$ or $G=\SO_n(F)$. There exists a well-defined surjective map $$\Scl : \begin{array}[t]{ccc}
\Phi_{e}(G) & \longrightarrow & \Omega_{e}^{\st}(G) \\
(\phi,\eta)  & \longmapsto & (\widehat{L}, \varphi,\varepsilon)
\end{array},$$ with the property that $\lambda_{\phi}=\lambda_{\varphi}$.
\end{theorem}

\begin{proof}
Here we give a sketch of the proof. Full details are available in {\cite[3.20]{Moussaoui:2015}}.

Recall the relation between the Langlands parameter in term of the Weil-Deligne group $W_F'$ and of the original Weil-Deligne group $WD_F=W_F \rtimes \C$.
A Langlands parameter for $G$ using the original Weil-Deligne group is a pair $(\lambda,N)$ with $\lambda : W_F \longrightarrow \widehat{G}$ an admissible morphism and $N\in \widehat{\mathfrak{g}}$ such that $$\forall w \in W_F, \,\, \Ad(\lambda(w))N=|w| N.$$
To $\phi : W_F' \longrightarrow \widehat{G}$ one can associate a pair $(\lambda,N)$ by $$\phi \longmapsto (\lambda_{\phi},N_{\phi}), \,\, \forall w \in W_F, \, \lambda_{\phi}=\phi(w,d_w) ,\,\, N_{\phi}=d \restriction{\phi}{\SL_2(\C)} \begin{pmatrix} 0 & 1 \\ 0 & 0\end{pmatrix}.$$ 
In the other direction, if $(\lambda,N)$ is fixed, by the Jacobson-Morozov-Kostant theorem, there exists a map $\gamma : \SL_2(\C) \longrightarrow \widehat{G}$ such that the differential of $\gamma$ sends $(\begin{smallmatrix} 0 & 1 \\ 0 & 0\end{smallmatrix})$ to $N$ and for all $t \in \C^{\times}$ and $\gamma(\diag(t,t^{-1}))$ commutes with the image of $\lambda$. Then, if we define for all $w \in W_F$, $\chi_{\phi}$ by $\chi_{\phi}(w)=\gamma(d_w)^{-1}$ then we set $$\phi(w,x)=\lambda(w)\chi_{\phi}(d_w)\gamma(x).$$

Now we need a construction which involves the Springer correspondence.
We apply the Springer correspondence for the group $H_{\phi}^{G}=Z_{\widehat{G}}(\restriction{\phi}{W_F})$, the unipotent class of $u_{\phi}=\phi\left(1,(\begin{smallmatrix} 1 & 1 \\ 0 & 1\end{smallmatrix}) \right) $, or more precisely to the nilpotent class of $N_{\phi}=d \restriction{\phi}{\SL_2(\C)} (\begin{smallmatrix} 0 & 1 \\ 0 & 0\end{smallmatrix})$ and the irreducible representation $\widetilde{\eta}$ of $A_{H_{\phi}^{G}}(u_{\phi})$. This defines a quasi-Levi subgroup $H'$ of $H_{\phi}^{G}$ and a nilpotent $N_{\varphi}$ element of the Lie algebra of $H'$. 

Remember that we want to define a cuspidal triple $(\widehat{L},\varphi,\varepsilon) \in \Omega_{e}^{\st}(G)$ such that $\lambda_{\varphi}=\lambda_{\phi}$.
Let $A=Z_{H'}^{\circ}$ be the identity component of the centre of $H'$ and let $\widehat{L}=Z_{\widehat{G}}(A)$. Then $\widehat{L}$ is a Levi subgroup of $\widehat{G}$. Since we have fixed $\lambda$ and we have obtained a nilpotent element $N_{\varphi}$, we have to check if this defines a Langlands parameter. By an adaptation of a result of Lusztig, for all $w \in W_F, \,\, \Ad(\lambda(w))N_{\varphi}=|w|N_{\varphi}$. Then we can define $\varphi : W_F' \longrightarrow \widehat{L}$ for all $(w,x) \in W_F'$ by $\varphi(w,x)=\lambda(w)\chi_{\varphi}(w) \gamma_{\varphi}(x)$. The nilpotent orbits which carry cuspidal local systems are distinguished. Hence $\varphi$ is a discrete parameter of $L$. It is automatically cuspidal because the Springer correspondence associates to $\varphi$ a cuspidal representation of $A_{\widehat{L}}(\varphi)$. 
\end{proof}

With reference to the proof above, note that, for all $w \in W_F$, $$\phi(w,1)=\lambda(w) \chi_{\phi}(w) \,\, \text{and} \,\, \varphi(w,1)=\lambda(w) \chi_{\varphi}(w).$$ Hence, $$\restriction{\phi}{W_F}=\restriction{\varphi}{W_F} \chi_{c}, \,\, \chi_c=\chi_{\phi}/\chi_{\varphi}.$$ We call $\chi_c$ a \emph{correcting cocharacter} of $\varphi$ in $\widehat{G}$.
This notion is treated with more detail in \cite[3.16,3.17]{Moussaoui:2015}.

The following proposition described the fibers of the map $
\Scl : \Phi_{e}(G) \to \Omega_{e}^{\st}(G)$ appearing in Theorem~\ref{theoremesupportcuspidal}. For $a,b \in \Z$, such that $a\leqslant b$, we denote by $\llbracket a,b \rrbracket$ the set of integers between $a$ and $b$.

\begin{prop}\label{propfiber}
Let $(\widehat{L},\varphi,\varepsilon) \in \Omega_{e}^{\st}(G)$ and $\chi_{c_1},\ldots,\chi_{c_r}$ be the correcting cocharacters of $\varphi$ in $\widehat{G}$. For each $i \in \llbracket 1,r \rrbracket$, we can consider the group $H_{\varphi \chi_{c_i}}^{G}=Z_{\widehat{G}}(\restriction{\varphi}{W_F} \chi_{c_i})$, its quasi-Levi subgroup $H_{\varphi \chi_{c_i}}^{L}=Z_{\widehat{L}}(\restriction{\varphi}{W_F} \chi_{c_i})$ and the relative Weyl group $$W_{H_{\varphi \chi_{c_i}}^{L}}^{H_{\varphi \chi_{c_i}}^{G}}=N_{H_{\varphi \chi_{c_i}}^{G}}(H_{\varphi \chi_{c_i}}^{L})/H_{\varphi \chi_{c_i}}^{L}.$$ The fiber of $\Scl $ above $(\widehat{L},\varphi,\varepsilon)$ is parametrized by the irreducible representations of $\Irr(W_{H_{\varphi \chi_{c_i}}^{L}}^{H_{\varphi \chi_{c_i}}^{G}})$ with $i \in \llbracket 1,r \rrbracket$ such that the parameter $\phi$ constructed as above satisfies $\chi_{c_i}=\chi_{\phi}/\chi_{\varphi}$.
\end{prop}

\begin{proof}
In the proof of Theorem~\ref{theoremesupportcuspidal}, there is an additional object which is needed to characterize $(\phi,\eta)$: the irreducible representation $\rho \in \Irr(W_{H_{\phi}^{L}}^{H_{\phi}^G})$ given by the Springer correspondence. Now we see that if $(\phi,\eta) \in \Phi_e(G)$ has cuspidal support $(\widehat{L},\varphi,\varepsilon)$ then necessarily $\restriction{\phi}{W_F}=\restriction{\varphi}{W_F} \chi_c$ with $\chi_c$ a correcting cocharacter. The set of correcting cocharacters of $\varphi$ in $\widehat{G}$ is finite (this can be deduced from \cite[5.4.c]{Kazhdan:1987}). Let $\chi_{c_1}, \ldots, \chi_{c_r}$ be the correcting cocharacters of $\varphi$ in $\widehat{G}$ and for all $i \in \llbracket 1,r \rrbracket$, let $\mu_i=\restriction{\varphi}{W_F} \chi_{c_i}$. Let $i \in \llbracket 1,r \rrbracket$ and consider an irreducible representation $\rho \in \Irr( W_{H_{\mu_i}^{L}}^{H_{\mu_i}^{G}})$. By the Springer correspondence for the group $H_{\mu_i}^{G}$, to $\rho$ is associated a unipotent element $u_{\mu_i,\rho} \in H_{\mu_i}^{G}$ or, equivalently, a morphism $\gamma_{(\mu_i,\rho)} : \SL_2(\C) \longrightarrow \left(H_{\mu_i}^{G}\right)^{\circ}$ and an irreducible representation $\eta$ of $A_{H_{\mu_i}^{G}}(\gamma_{(\mu_i,\rho)})$. Define $\phi_{(\mu_i,\rho)}=\mu_{i} \gamma_{(\mu_i,\rho)} : W_F' \longrightarrow \widehat{G}$. We can assume after conjugation that $\phi_{(\mu_i,\rho)}$ is adapted to $\varphi$ (see \cite[3.16]{Moussaoui:2015}). Now we apply the previous construction to see that $(\phi_{(\mu_i,\rho)},\eta) \in \Phi_e(G)$ is associated to $(\widehat{L},\varphi,\varepsilon)$ if and only if $\lambda_{\phi_{(\mu_i,\rho)}}=\lambda_{\varphi}$; in other words if and only if $\chi_{c_i}=\chi_{\phi_{(\mu_i,\rho)}}/\chi_{\varphi}$. 
\end{proof}

We saw at the beginning of Section~\ref{sectioncuspidallanglandsparameter} that we wanted to compare the two decomposition : $$\Irr(G)=\bigsqcup_{\mathfrak{s} \in \mathcal{B}(G)} \Irr(G)_{\mathfrak{s}} = \bigsqcup_{\phi \in \Phi(G)} \Pi_{\phi}(G).$$
For $\mathfrak{s}=[L,\sigma] \in \mathcal{B}(G)$, let $\Sil(\mathfrak{s})$ be the inertial pair $[\widehat{M}_{\lambda_{\varphi_\sigma}},\lambda_{\varphi_\sigma}] \in \mathcal{B}_{e}^{\st}(G)$, where $\varphi_\sigma : W_F' \longrightarrow \Phi(L)$ is the Langlands parameter of $\sigma$ and $\widehat{M}_{\lambda_{\varphi_\sigma}}$ is a Levi subgroup of $\widehat{G}$ which contains minimally the image of $\lambda_{\varphi_\sigma}$. We remark that if $\restriction{\varphi}{\SL_2} \neq 1$ then $L$ is not the dual of $\widehat{M}_{\lambda_{\varphi_\sigma}}$. We have proved the following : 

\begin{theorem}
Let $\wi=[\widehat{M},\lambda] \in \mathcal{B}^{\st}(G)$. Then we have : $$\Pi_{\wi}^{+}(G)=\bigsqcup_{\substack{\mathfrak{s} \in \mathcal{B}(G) \\ \Sil(\mathfrak{s})=\wi}} \Irr(G)_{s}.$$
\end{theorem}

This motivates the following conjecture.

\begin{conj}
Let $G$ be a reductive connected split group over $F$. Let $\wi=[\widehat{M},\lambda] \in \mathcal{B}^{\st}(G)$. Then, we have : $$\Pi_{\wi}^{+}(G)=\bigsqcup_{\substack{\mathfrak{s} \in \mathcal{B}(G) \\ \Sil(\mathfrak{s})=\wi}} \Irr(G)_{s}.$$
\end{conj}

\section{Aubert-Baum-Plymen-Solleveld conjecture for split classical groups}

\subsection{Aubert-Baum-Plymen-Solleveld conjecture}\label{sec:ABPSconj}

In this section we review the Aubert-Baum-Plymen-Solleveld conjecture as is stated in \cite[15]{Aubert:2014ac}. Let begin with the definitions of the so-called "extended quotient". Let $T$ be a complex affine variety and $\Gamma$ be a finite group acting on $T$ as automorphisms of affine variety. For all $t \in T$, let $\Gamma_t =\{ \gamma \in \Gamma \mid \gamma \cdot t = t \}$ be the stabilizer of $t$ in $\Gamma$. 
Consider $$X=\{(t,\gamma) \in T \times \Gamma \mid \gamma \cdot t=t\} \quad \text{and} \quad Y=\{(t,\rho) \mid t \in T, \rho \in \Irr(\Gamma_t)\}.$$ The group $\Gamma$ acts on $X$ and $Y$ by : $$\alpha \cdot (t,\gamma)=(\alpha \cdot t, \alpha \gamma\alpha^{-1}), \quad \text{and} \quad \alpha \cdot (t,\rho) =(\alpha \cdot t, \alpha^* \rho), \,\, \alpha \in \Gamma, \, (t,\rho) \in Y,$$ where $\alpha^* \rho \in \Irr(\Gamma_{\alpha \cdot t})$ is defined by, $(\alpha^* \rho)(\gamma)=\rho(\alpha \gamma \alpha^{-1})$, for all $\gamma \in \Gamma_{\alpha \cdot t}$. Remark that $X$ has a natural structure of affine variety whereas $Y$ does not admit a natural structure of algebraic variety. In the following we recall the definitions of the extended quotient as is stated in \cite[11,13]{Aubert:2014ac} but we give a different names.
\begin{definition}[{\cite[11,13]{Aubert:2014ac}}]
The \emph{geometric extended quotient of $T$ by $\Gamma$} is the quotient $X/\Gamma$ and it is denoted by $T \sslash \Gamma$. The \emph{spectral extended quotient of $T$ by $\Gamma$} is the quotient $Y/\Gamma$ and it is denoted by $T \sslash \widehat{\Gamma}$.
\end{definition}

Notice that in \cite{Aubert:2014ac} the authors state their conjecture with the hypothesis that $G$ is quasi-split. They have also a conjecture when $G$ is non necessarily quasi-split.

\begin{conj}[Aubert-Baum-Plymen-Solleveld]
Let $G$ be a split connected reductive $p$-adic group and $\mathfrak{s} \in \mathfrak{B}(G)$ be an inertial pair for $G$. Then 
\begin{enumerate}
\item 
The cuspidal support map $$\Sc : \Irr(G)_{\mathfrak{s}} \rightarrow T_{\mathfrak{s}}/W_{\mathfrak{s}}$$ is one-to-one if and only if the action of $W_{\mathfrak{s}}$ on $T_{\mathfrak{s}}$ is free.
\item 
There is a canonically defined commutative triangle 
\[
\begin{tikzcd}
{} & T_{\mathfrak{s}} \sslash \widehat{W_{\mathfrak{s}}} \arrow[swap]{ld}{\mu_{\mathfrak{s}}}\arrow{rd} & \\
 \Irr(G)_{\mathfrak{s}} \arrow{rr} & & \Phi(G)_{\mathfrak{s}}
\end{tikzcd}
\]
\end{enumerate}
Moreover, the bijection $\mu_{\mathfrak{s}}$ should satisfies the following properties: 
\begin{enumerate}[(i)]
\item The bijection $\mu_{\mathfrak{s}}$ maps $K_{\mathfrak{s}} \sslash \widehat{W_{\mathfrak{s}}}$ onto $\Irr(G)_{\mathfrak{s},\temp}$.
\item For many $\mathfrak{s} \in \mathcal{B}(G)$, the diagram \begin{center}
\begin{tikzcd}
  \Irr(G)_{\mathfrak{s}} \arrow[swap]{rd}{\Sc} \arrow[leftrightarrow]{rr}{\mu_{\mathfrak{s}}} & & T_{\mathfrak{s}} \sslash \widehat{W_{\mathfrak{s}}} \arrow{ld}{\proj_{\mathfrak{s}}} \\
   & T_{\mathfrak{s}} / W_{\mathfrak{s}} &
  \end{tikzcd}
\end{center}
does not commute.
\item There is an algebraic family $$\theta_z : T_{\mathfrak{s}} \sslash \widehat{W_{\mathfrak{s}}} \rightarrow T_{\mathfrak{s}} / W_{\mathfrak{s}}$$ of finite morphisms of algebraic varieties, with $z \in \C^{\times}$, such that $$\theta_{1}=\proj_{\mathfrak{s}}, \quad \theta_{\sqrt{q}}=\Sc \circ \mu_{\mathfrak{s}}$$
\item For each connected component $\mathbf{c}$ of the affine variety $T_{\mathfrak{s}} \sslash W_{\mathfrak{s}}$, there is a cocharacter $$h_{\mathbf{c}} : \C^{\times} \longrightarrow T_{\mathfrak{s}}$$ such that $$\theta_z[t,w]=W_{\mathfrak{s}}(h_{\mathbb{c}}(z) \cdot t) \in T_{\mathfrak{s}}/W_{\mathfrak{s}}, $$ for all $[t,w] \in \mathbf{c}$.\\ Let $Z_1,\ldots,Z_r$ be the connected components of the affine variety $T_{\mathfrak{s}} \sslash W_{\mathfrak{s}}$ and let $h_1,\ldots,h_r$ be the cocharacters associated. Let $$\nu_{\mathfrak{s}} : X_{\mathfrak{s}} \rightarrow T_{\mathfrak{s}} \sslash W_{\mathfrak{s}}$$ be the quotient map. Then the connected components $X_1,\ldots,X_r$ of the affine variety $X_{\mathfrak{s}}$ can be chosen with \begin{itemize}
\item $\mu_{\mathfrak{s}}(X_j)=Z_j$ for $j \in \llbracket 1,r \rrbracket$.
\item For each $z \in \C^{\times}$ the map $m_z : X_j \to T_{\mathfrak{s}} / W_{\mathfrak{s}}$, which is the composition $$ \begin{array}{lllll}
X_j & \to & T_{\mathfrak{s}} & \to & T_{\mathfrak{s}}/W_{\mathfrak{s}} \\
(t,w) & \mapsto & h_j(z)t & \mapsto & W_{\mathfrak{s}} (h_j(z)t)
\end{array}$$ makes the diagram \begin{center}
\begin{tikzcd}
  X_j \arrow[swap]{rd}{m_z} \arrow{rr}{\nu_{\mathfrak{s}}} & & Z_j \arrow{ld}{\theta_z} \\
   & T_{\mathfrak{s}} / W_{\mathfrak{s}} &
  \end{tikzcd}
\end{center}
\item There exists a map of sets $\lambda : {Z_1,\ldots,Z_r} \to V$ (called a labeling) such that for any two points $[t,w]$ and $[t',w']$ of $T_{\mathfrak{s}} \sslash W_{\mathfrak{s}}$: $\mu_{\mathfrak{s}}[t,w]$ and $\mu_{\mathfrak{s}}[t',w']$ are in the same $L$-packet if and only if $\theta_z[t,w]=\theta_z[t',w']'$ for all $z \in \C^{\times}$ and $\lambda[t,w]=\lambda[t',w']$, where $\lambda$ has been lifted to a labelling of $T_{\mathfrak{s}} \sslash W_{\mathfrak{s}}$ in the evident way.
\end{itemize}
\end{enumerate}

\end{conj}

Aubert, Baum and Plymen proved the conjecture for the group $G_2$ in \cite{Aubert:2011aa}. 
Solleveld proved a version of this conjecture for extended Hecke algebras in \cite{Solleveld:2012aa} which, as a consequence, demonstrates the validity of the ABPS for split classical groups. In a refined version stated in \cite{Aubert:2015}, Aubert, Baum, Plymen and Solleveld prove the conjecture for the inner forms of $\GL_n$ and $\SL_n$ using the relation with the Langlands correspondence. Recently, in \cite{Aubert:2014aa}, the authors prove the conjecture for the principal series representations of split connected reductive groups, in relation with the Langlands correspondence.

\subsection{Galois version of ABPS conjecture}

Let $G$ be a split classical group, {\it i.e.}, $G=\Sp_{2n}(F)$ or $G=\SO_n(F)$. Let $\wj=[\widehat{L},\varphi,\varepsilon] \in \mathcal{B}_e^{\st}(G)$. Recall the we have defined a torus $\mathcal{T}_{\wj}=\{(\varphi\chi)_{\widehat{L}} \mid \chi \in {}\mathcal{X}(\widehat{L})\}$. 
Since $\varphi$ is fixed and the multiplication by an unramified cocharacter does not affect the $\SL_2(\C)$ part, we can identify $\mathcal{T}_{\wj}$ with the restriction of $\varphi\chi$ to $W_F$ for all $\chi \in \mathcal{X}(\widehat{L})$. Moreover, if $(\phi,\eta) \in \Phi_e(G)_{\wj}$, we denote by $\rho_{(\phi,\eta)} \in \Irr(W_{H_{\phi}^L}^{H_{\phi}^G})$ the irreducible representation attached by the Springer correspondence.

\begin{theorem}\label{galoisABPS}
Let $G$ be a split classical group, {\it i.e.}, $G=\Sp_{2n}(F)$ or $G=\SO_n(F)$. Let $\wj=[\widehat{L},\varphi,\varepsilon] \in \mathcal{B}_e^{\st}(G)$.  Then the following map defines a bijection : $$\mu_{\wj} : \begin{array}[t]{ccc}
\Phi_{e}(G)_{\wj} & \longrightarrow & \mathcal{T}_{\wj}  \sslash \widehat{\mathcal{W}_{\wj}}\\
(\phi,\eta)  & \longmapsto & (\restriction{\phi}{W_F},\rho_{(\phi,\eta)})
\end{array}.$$ 
\end{theorem}

Just before proving the theorem, notice that the theorem is true without assuming the Langlands correspondence.

\begin{proof}
Let $(\phi,\eta) \in \Phi_{e}(G)_{\wj}$. Then $\restriction{\phi}{W_F}=\restriction{\varphi}{W_F} \chi \chi_c$, where $\chi \in \mathcal{X}(L)$ and $\chi_c$ is the correcting cocharacter associated to $(\phi,\eta)$. Hence $\restriction{\phi}{W_F}$ is a twist of $\restriction{\varphi}{W_F}$ by an unramified cocharacter. Denote by $A_{\widehat{L}}=Z_{\widehat{L}}^{\circ}$ and note that the stabilizer of $\restriction{\phi}{W_F}$ is \begin{align*}
\mathcal{W}_{\wj,\phi} &= \{ w \in \mathcal{W}_{\wj} \mid ({}^w(\varphi \chi \chi_c))_{\widehat{L}} = (\varphi\chi \chi_c)_{\widehat{L}}\} \\ 
& \simeq N_{Z_{\widehat{G}}(\varphi\chi \chi_c)}(A_{\widehat{L}})/Z_{\widehat{L}}(\varphi\chi \chi_c) \\
& = N_{Z_{Z_{\widehat{G}}(\restriction{\varphi}{W_F}\chi \chi_c)}(\restriction{\varphi}{\SL_2})}(A_{\widehat{L}})/Z_{Z_{\widehat{L}}(\restriction{\varphi}{W_F}\chi \chi_c)}(\restriction{\varphi}{\SL_2})\\
& = N_{Z_{Z_{\widehat{G}}(\restriction{\phi}{W_F})}(\restriction{\varphi}{\SL_2})}(A_{\widehat{L}})/Z_{Z_{\widehat{L}}(\restriction{\phi}{W_F})}(\restriction{\varphi}{\SL_2})\\
& \simeq N_{Z_{\widehat{G}}(\restriction{\phi}{W_F})}(A_{\widehat{L}})/Z_{\widehat{L}}(\restriction{\phi}{W_F})\\
& = W_{H_{\phi}^L}^{H_{\phi}^{G}}.
\end{align*} 
Here we use \cite[2.6.b]{Lusztig:1988aa} and Table \ref{exemplequasiLevi} in the penultimate line.  This shows that the map $\mu_{\wj} $ is well defined. This map is surjective by  Proposition \ref{propfiber} and its proof. Moreover, the bijectivity of the Springer correspondence for the groups $H_{\varphi\chi}^{G}$ shows that this map is injective.

\end{proof}

\subsection{Proof of ABPS conjecture}\label{sec:proof}

Let $G$ be a split classical group and $\wj=[\widehat{L},\varphi,\varepsilon] \in \mathcal{B}_e^{\st}(G)$. Before proving the ABPS conjecture, let us introduce some definitions and notations. We denote by $\Phi(G)_2$ (resp. $\Phi(G)_{\temp}$) the set of discrete (resp. tempered) Langlands parameters of $G$. By definition, $\phi \in \Phi(G)_2$ when $\phi(W_F)$ is not contained in a proper Levi subgroup of $\widehat{G}$ and $\phi \in \Phi(G)_{\temp}$ when $\phi(W_F)$ is bounded. 
Similarly, we denote by $\Phi_e(G)_{2}$ (resp. $\Phi_e(G)_{\temp}$) the set of enhanced Langlands parameters for which the Langlands parameter is discrete (resp. tempered). 
Recall that in $\Sp_{2n}(\C)$ or $\SO_n(\C)$ the unipotent classes are completely determined by their Jordan decomposition, or in other words, by the partition associated (except for $\SO_{2n}(\C)$ and when the partition has only even parts with even multiplicities for which there are two distincts orbits). 
Because the group that we will consider are products of complex symplectic groups, orthogonal groups and general linear groups, the unipotent classes which arise in $Z_{\widehat{G}}(\restriction{\varphi}{W_F} \chi)^{\circ}$ are characterized by their partition. 
In particular, as $\chi$ runs over $\mathcal{X}(\widehat{L})$, finitely many unipotent classes arise in this manner. 
Let $\mathcal{CU}$ be a system of representative of unipotent classes of $Z_{\widehat{G}}(\restriction{\varphi}{W_F} \chi)^{\circ}$ when $\chi$ runs over $\mathcal{X}(\widehat{L})$. 
We can assume that elements in $\mathcal{CU}$ are adapted to $\varphi$ in $\widehat{G}$ (see \cite[3.16]{Moussaoui:2015}). 
Let $u \in \mathcal{CU}$ and $\gamma_{u} : \SL_2(\C) \longrightarrow Z_{\widehat{G}}(\varphi(I_F))^{\circ}$ be such that $\gamma_{u}$ is adapted to $\restriction{\varphi}{\SL_2}$. Define $$c_{u} : \begin{array}[t]{ccc}
\C^{\times} & \longrightarrow &Z_{\widehat{L}}^{\circ} \\
z & \longmapsto & \gamma_{u}\left(\begin{smallmatrix} z & 0 \\ 0 & z^{-1}\end{smallmatrix}\right) / \restriction{\varphi}{\SL_2}\left(\begin{smallmatrix} z & 0 \\ 0 & z^{-1}\end{smallmatrix}\right). 
\end{array}$$

\begin{prop}
Let $G$ be a split classical group and suppose $\wj=[\widehat{L},\varphi,\varepsilon] \in \mathcal{B}_e^{\st}(G)$.
The map $\mu_{\wj}$ satisfies the following properties.
\begin{enumerate}
\item The cuspidal support map $$\Scl : \Phi_e(G)_{\wj} \rightarrow \mathcal{T}_{\wj}/\mathcal{W}_{\wj}$$ is one-to-one if and only if the action of $\mathcal{W}_{\wj}$ on $\mathcal{T}_{\wj}$ is free.
\item Let $\mathcal{K}_{\wj}$ be the maximal compact torus in $\mathcal{T}_{\wj}$. Then the previous bijection induces a bijection 
\[
\mathcal{K}_{\wj} \sslash \widehat{\mathcal{W}_{\wj}} \longleftrightarrow \Phi_e(G)_{\wj} \cap \Phi_e(G)_{\temp}.
\]
\item Let $\mathcal{CU}$ be a system of representatives of unipotent classes of $Z_{\widehat{G}}(\restriction{\varphi}{W_F} \chi)^{\circ}$, when $\chi$ runs over $\mathcal{X}(\widehat{L})$. There exists a partition of $\mathcal{T}_{\wj} \sslash \widehat{\mathcal{W}_{\wj}}$ indexed by $\mathcal{CU}$ with the following properties.
\begin{enumerate}[(i)]
\item $\displaystyle \mathcal{T}_{\wj} \sslash \widehat{\mathcal{W}_{\wj}}=\bigsqcup_{u \in \mathcal{CU}} \left(\mathcal{T}_{\wj} \sslash \widehat{\mathcal{W}_{\wj}} \right)_{u}$ (namely a point $(t,\rho) \in \left(\mathcal{T}_{\wj} \sslash \widehat{\mathcal{W}_{\wj}} \right)_{u}$ if and only if $u$ is the unipotent class associated by the Springer correspondence to $\rho$).

\item We have a bijection 
\[
\bigsqcup_{\substack{\mathcal{U} \in \mathcal{CU} \\ u \text{  distinguished orbit}}} \left( \mathcal{T}_{\wj} \sslash \widehat{\mathcal{W}_{\wj}} \right)_{u} \longleftrightarrow \Phi_e(G)_{\wj} \cap\Phi_e(G)_{2}.
\]

\item For $z \in \C^{\times}$, define $$\theta_z: \mathcal{T}_{\wj} \sslash \widehat{\mathcal{W}_{\wj}} \longrightarrow \mathcal{T}_{\wj} / \mathcal{W}_{\wj} $$ by  $\theta_z(t,\rho)=\mathcal{W}_{\wj} \cdot (c_u(z) t)$ if $(t,\rho) \in \left(  \mathcal{T}_{\wj} \sslash \widehat{\mathcal{W}_{\wj}} \right)_{u}$. Then $$\theta_1 = \mathbf{p}_{\wj}, \qquad \text{and}\qquad \Scl= \theta_{\sqrt{q}}\circ \mu_{\wj}.$$
\item Let $u,v \in \mathcal{CU}$, $(t,\rho) \in \left(\mathcal{T}_{\wj} \sslash \widehat{\mathcal{W}_{\wj}} \right)_{u}$ and $(t',\rho') \in \left(\mathcal{T}_{\wj} \sslash \widehat{\mathcal{W}_{\wj}} \right)_{v}$. Then $\mu_{\wj}^{-1}(t,\rho)$ and $\mu_{\wj}^{-1}(t',\rho')$ have the same Langlands parameter if and only if $u=v$ and for all $z \in \C^{\times}, \,\, \theta_z(t,\rho)=\theta_z(t',\rho')$.
\end{enumerate}
\end{enumerate}
\end{prop}

\begin{proof}
In Theorem \ref{galoisABPS} we proved that we have a bijection between $\Phi_e(G)_{\wj}$ and the extended quotient $\mathcal{T}_{\wj} \sslash \widehat{\mathcal{W}_{\wj}}$. Hence, $\Scl$ is a bijection if and only if there is a bijection between $\mathcal{T}_{\wj} \sslash \widehat{\mathcal{W}_{\wj}}$ and $\mathcal{T}_{\wj} / \mathcal{W}_{\wj}$. 
The last statement is equivalent to saying that $\mathcal{W}_{\wj}$ acts freely on $\mathcal{T}_{\wj}$. 
By definition of the map $\mu_{\wj}$, the restriction to $W_F$ of the Langlands parameter associated to a point $(\mu,\rho) \in \mathcal{T}_{\wj} \sslash \widehat{\mathcal{W}_{\wj}}$ is $\mu$. Hence, $(\mu,\rho) \in \mathcal{K}_{\wj} \sslash \widehat{\mathcal{W}_{\wj}}$ if and only if $\mu(W_F)$ is bounded, if and only if $\mu_{\wj}^{-1}(\mu,\rho) \in \Phi_e(G)_{\wj,\temp}$. For point $(i)$: the definition made in the proposition defines the partition. 
For point $(ii)$: a Langlands parameter $\phi$ of $G$ is discrete if and only if $\phi(1,(\begin{smallmatrix} 1 & 1 \\ 0 & 1 \end{smallmatrix}))$ defines a distinguished unipotent class of $H_{\phi}^G$. By the construction of $\mu_{\wj}$ and the partition defined in $(i)$, this shows $(ii)$. For point $(iii)$ is a consequence of the definition of the cuspidal support of an enhanced Langlands parameter. To conclude, for point $(iv)$,  if $u=v$ and if for all $z \in \C^{\times}, \,\, \theta_z(t,\rho)=\theta_z(t',\rho')$, then for $z=1$ we obtain $t=t'$. Recall that $t$ represents the restriction to $W_F$ of the Langlands parameter associate to the point. Since the points $(t,\rho)$ and $(t',\rho')$ have the same labelling $u$, their Langlands parameters have the same restriction to $\SL_2$, hence they have the same Langlands parameter. The other direction is evident by the definitions.
\end{proof}

\begin{theorem}\label{thm:main}
Let $G$ be a split classical group and let $\mathfrak{s}=[L,\sigma]$ be an inertial pair. Then there exists a bijection $$\Irr(G)_{\mathfrak{s}} \longleftrightarrow T_{\mathfrak{s}} \sslash \widehat{W_{\mathfrak{s}}},$$ which satisfies the same properties described above by replacing the corresponding object on the side of representation theory. 
\end{theorem}

\begin{proof}
In \cite[4.1]{Moussaoui:2015} we proved that if $\mathfrak{s}=[L,\sigma] \in \mathcal{B}(G)$ is an inertial pair with $L$ a Levi subgroup of $G$ and if $\sigma$ is an irreducible supercuspidal representation of $L$ and if $\wj=[\widehat{L},\varphi,\varepsilon] \in \mathcal{B}_e^{\st}(G)$ is the corresponding inertial triple obtained by the local Langlands correspondence, then $T_{\mathfrak{s}} \simeq \mathcal{T}_{\wj}, \,\, W_{\mathfrak{s}} \simeq \mathcal{W}_{\wj}$ and the action of $\mathcal{W}_{\wj}$ on $\mathcal{T}_{\wj}$ corresponds to the action of $W_{\mathfrak{s}}$ on $T_{\mathfrak{s}}$ through the previous isomorphisms. In particular, we have a natural bijection $$T_{\mathfrak{s}} \sslash \widehat{W_{\mathfrak{s}}} \longleftrightarrow \mathcal{T}_{\wj} \sslash \widehat{\mathcal{W}_{\wj}}.$$
Moreover, in theorem \ref{galoisABPS} we have seen that there is a bijection $$\mathcal{T}_{\wj} \sslash \widehat{\mathcal{W}_{\wj}} \longleftrightarrow \Phi_e(G)_{\wj}.$$ Notice that we need the local Langlands correspondence for supercuspidal representations, which is given by the work of Arthur. Finally, \cite[4.6]{Moussaoui:2015} shows that $\Irr(G)_{\mathfrak{s}}$ is in bijection with $\Phi_e(G)_{\wj}$. By composing these three bijections we obtain a proof of the Aubert-Baum-Plymen-Solleveld conjecture for classical groups.
\end{proof}

\begin{landscape}
\begin{figure}
\begin{center}
\begin{tikzpicture}[scale=0.7,every node/.style={minimum size=1cm},on grid]

\begin{scope}[
		yshift=-300,
		every node/.append style={yslant=\yslant,xslant=\xslant},
		yslant=\yslant,xslant=\xslant
	] 
\fill[white,fill opacity=.8] (-3,-3) rectangle (3,3);
\draw[black, dashed] (-3,-3) rectangle (3,3); 
\draw[->,black] (0,0) -- (1,0);
\draw[->,black] (0,0) -- (0,1);		

\draw[step=1.0,black,very thin,dashed] (-3,-3) grid (3,3);

\draw[fill=black] (1,-1) circle (.03); 

\coordinate (projq1q2) at (1,-1);	
\coordinate (projstein) at (-0.866025,-2.59808);
\coordinate (projdelta2) at (-0.333333,-1);
	\end{scope}
	
		\begin{scope}[
		yshift=-200,
		every node/.append style={yslant=\yslant,xslant=\xslant},
		yslant=\yslant,xslant=\xslant
	] 
\fill[white,fill opacity=.8] (-3,-3) rectangle (3,3);
		\draw[black!10!blue, dashed] (-3,-3) rectangle (3,3); 
    \draw[->,black!10!blue] (0,0) -- (1,0);
    \draw[->,black!10!blue] (0,0) -- (0,1);		
		\draw[black!10!blue,ultra thick] (-3,1)--(3,1);
		\draw[black!10!blue,ultra thick] (-3,-1)--(3,-1);		
		\draw[fill=black!10!blue] (1,-1) circle (.05); 
			
\coordinate (temp) at (-2,1);
\coordinate (temp2) at (-2,-1);
\coordinate (q1q2) at (1,-1);

	\end{scope}
	
	\begin{scope}[
		yshift=-100,
		every node/.append style={yslant=\yslant,xslant=\xslant},
		yslant=\yslant,xslant=\xslant
	] 
       \fill[white,fill opacity=.8] (-3,-3) rectangle (3,3);
		\draw[black!30!green, dashed] (-3,-3) rectangle (3,3); 
    \draw[->,black!30!green] (0,0) -- (1,0);
    \draw[->,black!30!green] (0,0) -- (0,1);		
		
		\draw[black!30!green,ultra thick] (-3,-3)--(3,3);
\coordinate (stein) at (-1.5,-1.5);
	\end{scope}
	
		\begin{scope}[
		yshift=0,
		every node/.append style={yslant=\yslant,xslant=\xslant},
		yslant=\yslant,xslant=\xslant
	] 
       \fill[white,fill opacity=.8] (-3,-3) rectangle (3,3);
		\draw[black!10!red, dashed] (-3,-3) rectangle (3,3); 
\draw[->,black!10!red] (0,0) -- (1,0);
\draw[->,black!10!red] (0,0) -- (0,1);
\draw[fill=black!10!red]  
(1,1) circle (.05)
(-1,-1) circle (.05); 
\coordinate (delta) at (1,1);
\coordinate (delta2) at (-1,-1);	
\end{scope}
	
		\draw[-latex,black!10!red] 
			(5,4.9) to[out=180,in=90] (delta);        
\draw[black!10!red](5,5.2)node[right]{$((1,1),\rho_{\{1^2,\varnothing\}} \boxtimes 1) \leftrightarrow \delta(\zeta)$};		
\draw[black!10!red](5,4.6)node[right]{$((1,1),\rho_{\{1^2,\varnothing\}} \boxtimes \varepsilon) \leftrightarrow \delta'(\zeta)$};	

\draw[black!10!red] (13,4.9) node[right,text width=7cm]{$\zeta \boxtimes (S_3\oplus S_1) \oplus 1$ (L-packet=$\{ \delta(\zeta\xi),\delta'(\zeta\xi),\sigma_{\zeta},\sigma_{\zeta}'\}$ with $\sigma_{\zeta},\sigma_{\zeta}'$ supercuspidal)};	
\draw[black!10!red] (13,4.9)--(11.3,4.9);

		\draw[-latex,black!10!red] 
			(5,2.9) to[out=180,in=30] (delta2);        
\draw[black!10!red](5,3.2)node[right]{$((-1,-1),\rho_{\{1^2,\varnothing\}} \boxtimes 1) \leftrightarrow \delta(\zeta\xi)$};		
\draw[black!10!red](5,2.6)node[right]{$((-1,-1),\rho_{\{1^2,\varnothing\}} \boxtimes \varepsilon) \leftrightarrow \delta'(\zeta\xi)$};	

\draw[black!10!red] (13,2.9)node[right,text width=7cm]{$\zeta\xi \boxtimes (S_3\oplus S_1) \oplus 1$ (L-packet=$\{ \delta(\zeta\xi),\delta'(\zeta\xi),\sigma_{\zeta\xi},\sigma_{\zeta\xi}'\}$ with $\sigma_{\zeta\xi},\sigma_{\zeta\xi}'$ supercuspidal)};	
\draw[black!10!red] (13,2.9)--(12.2,2.9);

		\draw[-latex,black!30!green] 
			(5,-2) to[out=180,in=30] (stein);        
\draw[black!30!green](5,-2)node[right]{$((z,z),\rho_{(1^2)}) \leftrightarrow \chi\zeta \mathrm{St}_{\mathrm{GL}_2} \rtimes 1$};	

\draw[black!30!green](13,-2)node[right]{$\chi\zeta \boxtimes S_2 \oplus 1 \oplus \chi^{-1} \zeta \boxtimes S_2$};	
\draw[black!30!green] (13,-2)--(11.5,-2);

\draw[-latex,black!10!blue] (5,-4.5) to[out=180,in=-30] (temp);
\draw[-latex,black!10!blue] (5,-6.1) to[out=180,in=-30] (temp2);
\draw[-latex,black!10!blue] (5,-8.1) to[out=180,in=-30] (q1q2);
                
\draw[black!10!blue](5,-4.2)node[right]{$((z,1),1) \leftrightarrow \chi\zeta \rtimes T_1^{\zeta}$};	
\draw[black!10!blue](5,-4.8)node[right]{$((z,1),\varepsilon) \leftrightarrow \chi\zeta \rtimes T_2^{\zeta}$};	

\draw[black!10!blue] (13,-4.5)node[right]{$\chi\zeta \oplus \zeta \oplus 1 \oplus\zeta \oplus \chi^{-1}\zeta$};		
\draw[black!10!blue] (13,-4.5)--(10,-4.5);

\draw[black!10!blue](5,-5.8)node[right]{$((z,-1),1) \leftrightarrow \chi\zeta \rtimes T_1^{\xi \zeta}$};	
\draw[black!10!blue](5,-6.4)node[right]{$((z,-1),\varepsilon) \leftrightarrow \chi\zeta \rtimes T_2^{\xi \zeta}$};	

\draw[black!10!blue] (13,-6.1) node[right]{$\chi\zeta \oplus \zeta\xi \oplus 1 \oplus \xi\zeta \oplus \chi^{-1}\zeta$};		
\draw[black!10!blue] (13,-6.1) --(10.6,-6.1);

\draw[black!10!blue](5,-7.2)node[right]{$((1,-1),1\boxtimes 1) \leftrightarrow Q_1(\zeta \rtimes T_1^{\xi \zeta})$};	
\draw[black!10!blue](5,-7.8)node[right]{$((1,-1),1\boxtimes \varepsilon) \leftrightarrow Q_2(\zeta \rtimes T_1^{\xi \zeta})$};
\draw[black!10!blue](5,-8.4)node[right]{$((1,-1),\varepsilon\boxtimes 1) \leftrightarrow Q_1(\zeta \rtimes T_2^{\xi \zeta})$};	
\draw[black!10!blue](5,-9)node[right]{$((1,-1),\varepsilon\boxtimes\varepsilon) \leftrightarrow Q_2(\zeta\rtimes T_2^{\xi \zeta})$};		

\draw[black!10!blue] (13,-8.1)node[right]{$\zeta \oplus \xi \zeta \oplus 1 \oplus \xi \zeta \oplus \zeta$};		

\draw[black!30!green,->] (stein) to (projstein);
\draw[black!10!blue,->] (q1q2) to (projq1q2);
\draw[black!10!red,->] (delta2) to (projdelta2);
	
\end{tikzpicture}
\end{center}
\caption{Extended quotient for $\Sp_4(F)$}\label{pretty picture}
\end{figure}
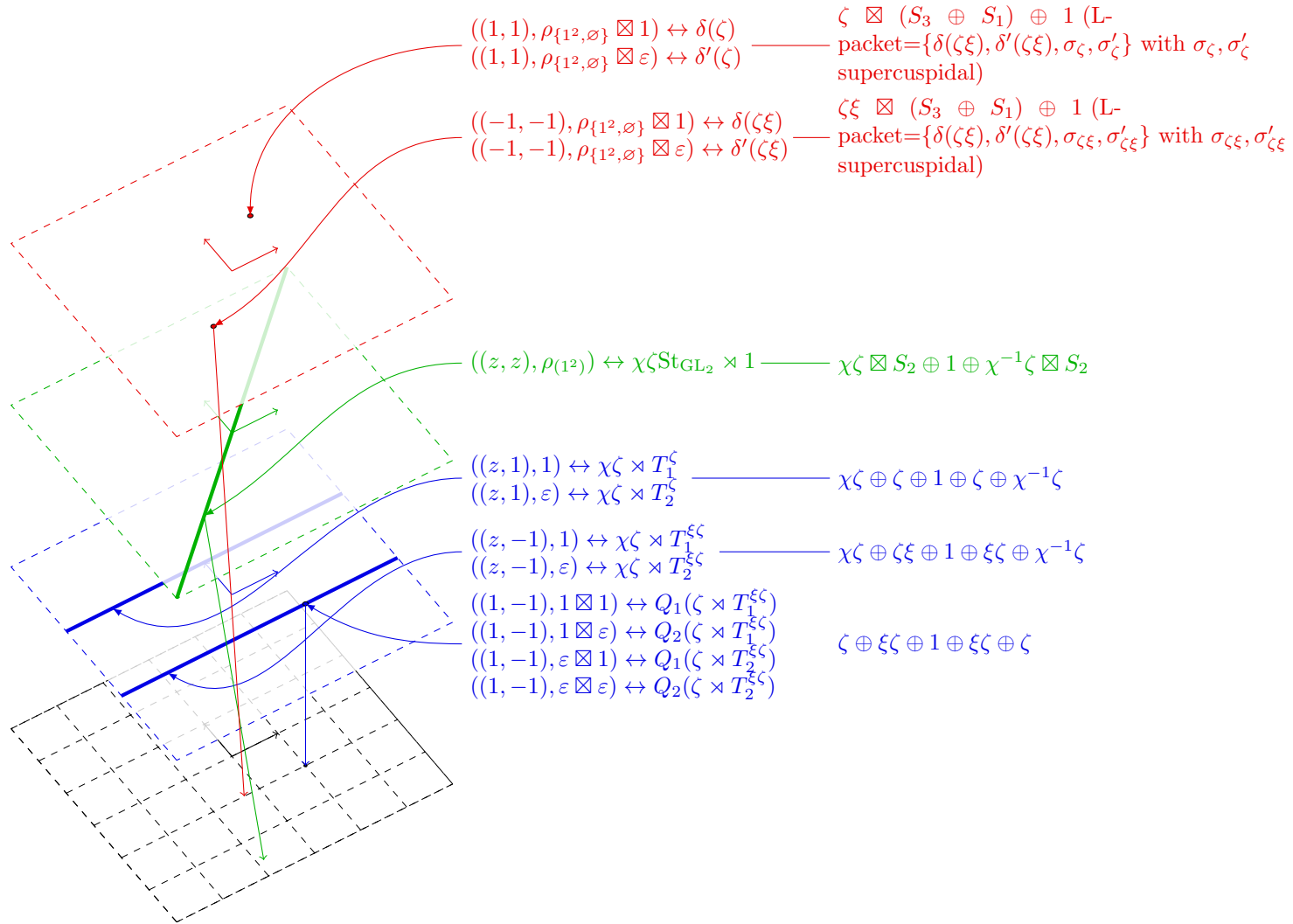
\end{landscape}

\begin{example}\label{exampleABPS}
We give here an example of Theorem~\ref{thm:main} in the case where $G=\Sp_4(F)$, $T=(F^{\times})^{2}$ is a maximal torus, $\zeta : F^{\times} \rightarrow \C^{\times}$ a ramified character and $\mathfrak{s}=[T,\zeta\boxtimes \zeta]$. The inertial pair $\mathfrak{s}$ corresponds to the inertial $L$-triple $\wj=[\widehat{T},\zeta \oplus \zeta, 1 ]_{\widehat{G}}$ (we use the same letter $\zeta$ for the character $\zeta$ and its Langlands parameter). We are looking at $$\Irr(G)_{\mathfrak{s}}= \{ \mbox{irreducible subquotients of } i_{B}^{G}(\chi_1 \zeta \boxtimes \chi_2 \zeta), \chi_1 \boxtimes \chi_2 \in  \mathcal{X}(T)\}.$$ The torus associated to $\mathfrak{s}$ is $T_{\mathfrak{s}}=\{\chi_1 \zeta \boxtimes \chi_2 \zeta, \chi_1 \boxtimes \chi_2 \in  \mathcal{X}(T)\}$. We have an isomorphism $T_{\mathfrak{s}} \simeq (\C^{\times})^2$ given by sending the character $\chi_1 \zeta \boxtimes \chi_2 \zeta \in T_{\mathfrak{s}}$ to the point $(z_1,z_2) \in (\C^{\times})^2$ where $(z_1,z_2)=((\chi_1 \zeta)(\Fr),(\chi_2\zeta)(\Fr))$. In this case $W_{\mathfrak{s}} \simeq N_G(T)/T \simeq \mathfrak{S}_2 \rtimes (\Z/2\Z)^2 \simeq \langle s_1 \rangle \rtimes (\langle s_2 \rangle \times \langle s_1 s_2 s_1 \rangle )=\langle s_1,s_2 \rangle $ and the action of $W_{\mathfrak{s}}$ on $T_{\mathfrak{s}}$ is described in the following table: 
\begin{table}[H]
\begin{center}
\renewcommand\arraystretch{1.2}
$$\begin{array}{|c|c|}
\hline
w & (z_1,z_2) \\ \hline
1 & (z_1,z_2)  \\ \hline
s_1 & (z_2,z_1) \\ \hline
s_2 & (z_1,z_2^{-1}) \\ \hline
s_1 s_2 & (z_2^{-1},z_1) \\ \hline
s_2 s_1 & (z_2,z_1^{-1}) \\ \hline
s_1 s_2 s_1 & (z_1^{-1},z_2) \\ \hline
s_2 s_1 s_2 & (z_2^{-1},z_1^{-1}) \\ \hline
s_1 s_2 s_1 s_2 & (z_1^{-1},z_2^{-1}) \\ \hline
\end{array} $$
\end{center}
\caption{Action of $W_{\mathfrak{s}}$ on $T_{\mathfrak{s}}$}
\label{table:actionWonT}
\end{table}
A set of representatives of the conjugacy classes of $W_{\mathfrak{s}}$ is $\{1,s_1,s_2,s_1s_2,s_1s_2s_1s_2\}$ and we have:
\begin{small}
\begin{align*}
T_\mathfrak{s}^{s_1} &= \{ (z,z) ,  z \in \C^\times \} & T_\mathfrak{s}^{s_1}/Z_\mathfrak{s}^{s_1}  &= \left\{ \left[ (z,z),(z^{-1},z^{-1}) \right] , z \in \C^\times \right\} \\
T_\mathfrak{s}^{s_2} &= \{ (z,1),(z,-1) ,  z \in \C^\times \} & T_\mathfrak{s}^{s_2}/Z_\mathfrak{s}^{s_2}  &= \left\{ \left[ (z,1),(z^{-1},1) \right] , z \in \C^\times \right\} \\
& & & \sqcup \left\{ \left[ (z,-1),(z^{-1},-1) \right] , z \in \C^\times \right\} \\
T_\mathfrak{s}^{s_1s_2} &= \{ (1,1), (-1,-1) \}  & T_\mathfrak{s}^{s_1s_2}/Z_\mathfrak{s}^{s_1 s_2}  &= \left\{ (1,1), (-1,-1) \right\}\\ 
T_\mathfrak{s}^{s_1 s_2 s_1 s_2} &= \{ (1,1), (1,-1), (-1,1), (-1,-1)  \} & T_\mathfrak{s}^{s_1 s_2 s_1 s_2}/Z_\mathfrak{s}^{s_1 s_2 s_1 s_2}  &= \left\{ (1,1), \left[ (1,-1), (-1,1) \right], (-1,-1)  \right\} 
\end{align*}
\end{small}
Hence we have a set of representatives of $Y_{\mathfrak{s}}$ (see \textsection \ref{sec:ABPSconj}) to describe the spectral extended quotient :
$$W_{\mathfrak{s},(z,z)}=\langle s_1 \rangle \simeq \mathfrak{S}_2, \quad W_{\mathfrak{s},(z,\pm 1)}=\langle s_2 \rangle \simeq \Z/2\Z, \quad W_{\mathfrak{s},(1,-1)}=\langle s_2 \rangle \times \langle s_1 s_2 s_1 \rangle  \simeq (\Z/2\Z)^2$$ $$ W_{\mathfrak{s},(1,1)}=W_{\mathfrak{s},(-1,-1)}=W_{\mathfrak{s}}=\mathfrak{S}_2 \rtimes (\Z/2\Z)^2$$
 
\begin{landscape}
\begin{table}[H]
\renewcommand\arraystretch{1.2}
$$\begin{array}{|c|c|c|c|c|c|c|c|}
\hline
\phi & H_{\phi}^{G} &(H_{\phi}^{G})^{\circ} & u_{\phi} & A_{{(H_{\phi}^{G})}^{\circ}}(u_{\phi})  & \Irr(A_{{(H_{\phi}^{G})}^{\circ}}(u_{\phi})) & H_{\varphi}^{L} &  \widehat{L} \\ \hline\hline
\multirow{2}*{$\zeta\omega \boxtimes (S_3 \oplus S_1) \oplus 1$} & \multirow{2}*{$S(\O_4 \times \O_1)$} & \multirow{2}*{$\SO_4$}  & \multirow{2}*{$(3,1)$} & \Z/2\Z  & \zeta & S(\O_4 \times \O_1)  &  \SO_5 \\ \cline{6-8} 

 &  &  & & \simeq \langle z_1z_3 \rangle   & 1 & \multirow{4}*{$\GL_1^2$} & \multirow{4}*{$\GL_1^2$} \\ \cline{1-6}

\chi \zeta \boxtimes S_2 \oplus 1 \oplus \chi^{-1} \zeta \boxtimes S_2  & \GL_2 & \GL_2  & (2) & \{1\} & 1 &  &  \\ \cline{1-6}

\chi \zeta \oplus \zeta \omega \oplus 1 \oplus \zeta \omega \oplus \chi^{-1} \zeta & \GL_1 \times S(\O_2 \times \O_1) &\GL_1 \times \SO_2 &  (1) \times (1) & \{1\} & 1 & & \\ \cline{1-6}

\zeta \oplus \zeta \xi \oplus 1 \oplus \zeta \xi \oplus \zeta & S(\O_2 \times \O_2 \times \O_1) & \SO_2 \times \SO_2 & (1) \times (1) & \{1\} & 1 & &  \\ \hline

\end{array} $$
\caption{Determining cuspidal supports}
\label{table:Langlandsparameter}
\end{table}
\begin{table}[H]
\begin{center}
\renewcommand\arraystretch{1.2}
$$\begin{array}{|c|c|c|c|c|c|}
\hline
\phi & A_{H_{\phi}^{G}}(u_{\phi}) & \Irr(A_{H_{\phi}^{G}}(u_{\phi})) & \Pi_{\phi}(G) & \Irr(\mathcal{W}_{\wj,\phi}) & \mathcal{W}_{\wj,\phi}=\mathcal{W}_{\wj,\phi}^{\circ} \rtimes \mathcal{R}_{\wj,\phi} \\ \hline \hline
\multirow{4}*{$\zeta\omega \boxtimes (S_3 \oplus S_1) \oplus 1$} & & \zeta \boxtimes 1 & \sigma_{\zeta \omega} & 1 & \{1\} \rtimes (\Z /2 \Z )   \\ \cline{3-5} 
 & (\Z/2\Z)^2 & \zeta \boxtimes \zeta & \sigma_{\zeta \omega}' & \zeta &\simeq \{1\} \rtimes \langle s_2 \rangle   \\ \cline{3-6} 
 
 & \simeq \langle z_1 z_3 \rangle \times \langle z_3 z_1' \rangle& 1 \boxtimes 1 & \delta(\zeta \omega) & \rho_{ \{1^2, \varnothing \} } \boxtimes 1 & (\mathfrak{S}_2 \rtimes \Z/2\Z) \rtimes (\Z /2 \Z )   \\  \cline{3-5} 
  & & 1 \boxtimes \zeta & \delta'(\zeta \omega) & \rho_{ \{1^2, \varnothing \} } \boxtimes \zeta &\simeq (\langle s_1 \rangle \rtimes \langle s_2s_1s_2 \rangle) \rtimes \langle s_2 \rangle \\ \hline
 
 \chi \zeta \boxtimes S_2 \oplus 1 \oplus \chi^{-1} \zeta \boxtimes S_2  & \{1\} & 1 & \chi \zeta \St_{\GL_2} \rtimes 1 & \rho_{(1^2)} & \mathfrak{S}_2 \rtimes \{1 \} \simeq \langle s_1 \rangle  \\ \hline
 
\multirow{2}*{$\chi \zeta \oplus \zeta \omega \oplus 1 \oplus \zeta \omega \oplus \chi^{-1} \zeta$} & \Z/2\Z & 1 &\chi \zeta \rtimes T_1^{\zeta \omega} & 1 & \{1\} \rtimes (\Z/2\Z) \\  \cline{3-5} 
& \simeq \langle z_1z_1' \rangle & \zeta &\chi \zeta \rtimes T_2^{\zeta \omega} & \zeta & \simeq \langle s_2 \rangle \\ 
\hline 
 
\multirow{4}*{$\zeta \oplus \zeta \xi \oplus 1 \oplus \zeta \xi \oplus  \zeta$} &  & 1 \boxtimes 1 & Q_1( \zeta \rtimes T_1^{\zeta \xi}) &  1 \boxtimes 1 & \\  \cline{3-5} 
&  (\Z/2\Z)^2  & 1 \boxtimes \zeta & Q_2( \zeta \rtimes T_1^{\zeta \xi}) & 1 \boxtimes \zeta & \{1\} \rtimes (\Z/2\Z)^2  \\ \cline{3-5} 
& \simeq \langle z_1z_1'' \rangle  \times \langle z_1'z_1'' \rangle & \zeta \boxtimes 1 & Q_1( \zeta \rtimes T_2^{\zeta \xi}) & \zeta \boxtimes 1 & \simeq \langle s_2 \rangle  \times \langle s_1s_2s_1s_2 \rangle \\ \cline{3-5} 
&  & \zeta \boxtimes \zeta & Q_2( \zeta \rtimes T_2^{\zeta \xi}) & \zeta \boxtimes \zeta &  \\ 
\hline

\end{array} $$
\end{center}
\caption{Parametrize the fibers}
\label{table:Langlandsparameter2}
\end{table}
\end{landscape}

As in the example \ref{ordspringer} the irreducible representations of $W_{\mathfrak{s}}$, which is a Weyl group of type $B_2$, are parametrized by bipartition of $2$ : $(2,\varnothing),(1^2,\varnothing),(1,1),(\varnothing,1^2),(\varnothing,2)$.\\

In the previous tables we determine, by the method that we have described above, the $L$-inertial packet $\Pi_{\wi}^{+}(G)$. The table \ref{table:Langlandsparameter} describes all Langlands parameters involved in this $L$-inertial packet, the groups $H_{\phi}^{G}$, the unipotents $u_{\phi}$, the groups of components $A_{(H_{\phi}^{G})^{\circ}}(u_{\phi}) $ and the cuspidal support. The table \ref{table:Langlandsparameter2} describes the $L$-packet of each Langlands parameters, the stabilizer group $\mathcal{W}_{\wj,\phi}$ and the irreducible representation attached to each representation that we have obtained by the Springer correspondence. Note here that we have twisted the Springer correspondence by the sign character so that for instance the Steinberg representation of $\GL_2$ is parametrized by the sign character.\\ We have denoted by $\omega$ the trivial character or an unramified character of order two $\xi$.\\ For the Langlands parameter $\phi=\zeta\omega \boxtimes (S_3 \oplus S_1) \oplus 1$, we find that $\mathcal{W}_{\wj,\phi}^{\circ}$ is a Weyl group of type $D_2$. The irreducible representation of the Weyl group of type $D_{n}$ are described by the sets $\{\alpha,\beta\}$ with $\alpha,\beta$ two partitions (perhaps trivial) such that $|\alpha| + |\beta|=n$. We know that we can embed the Weyl group of type $D_n$ in the Weyl group of type $B_n$ and for example, $\rho_{ \{1^2, \varnothing \} } \boxtimes 1$ (resp. $\rho_{ \{1^2, \varnothing \} } \boxtimes \zeta $)  correspond to the representation $\rho_{ (1^2, \varnothing) } $ (resp. $\rho_{ (\varnothing,1^2) } $).\\

Following an idea from Plymen, in Figure~\ref{pretty picture} we picture the extended quotient $T_{\mathfrak{s}} \sslash \widehat{W_{\mathfrak{s}}}$ with the decomposition with respect of the unipotent classes. 
In particular, \color{black!10!red} the plane in red is associated to the unipotent with partition $(3,1)$, \color{black!30!green} the plane in green is associated to the unipotent with partition $(2,2)$ and \color{black!10!blue} the plane in blue is associated with the partition $(1^4)$. \color{black} In particular, the last plane in black is where the usual quotient $T_{\mathfrak{s}}/W_{\mathfrak{s}}$ lives. We describe each point of the extended quotient, the corresponding representation (in the notation of \cite{Sally:1993}) and its Langlands parameter.

The $L$-inertial pair $\wi \in \mathcal{B}^{\st}(G)$ image of $\wj$ by $\mathcal{B}_e^{\st}(G) \rightarrow \mathcal{B}^{\st}(G)$ is $\wi=[\widehat{T},\zeta \oplus \zeta]$. Provided that $p \neq 2$ (otherwise there is more unramified characters of order $2$), we have: $$\Pi_{\wi}^{+}(G)=\Irr(G)_{[T,\zeta \boxtimes \zeta]} \sqcup \Irr(G)_{[G, \sigma_{\zeta}]} \sqcup \Irr(G)_{[G, \sigma_{\zeta\xi}]} \sqcup \Irr(G)_{[G, \sigma_{\zeta}']} \sqcup \Irr(G)_{[G, \sigma_{\zeta\xi}']}.$$

\end{example}

\printbibliography

\end{document}